\documentclass[reqno,a4paper,twoside,english]{amsart}

\usepackage{quoting}

\usepackage[cmyk,usenames,dvipsnames,svgnames]{xcolor}

\usepackage{amssymb,dsfont,mathrsfs,verbatim,bm,mathtools, float, booktabs,amsrefs}
\usepackage{babel,changes}
\usepackage{microtype}
\usepackage{enumitem}
\usepackage{graphicx}
\usepackage[figuresright]{rotating}
\newenvironment{amssidewaysfigure}
  {\begin{sidewaysfigure}\vspace*{.5\textwidth}\begin{minipage}{\textheight}\centering}
  {\end{minipage}\end{sidewaysfigure}}

\usepackage{float}
\usepackage[section]{placeins}

\usepackage{tikz, tikz-cd}
\usetikzlibrary{matrix}
\usetikzlibrary{shapes}
\usepackage{bm}

\usepackage{lmodern}

\usepackage[section]{placeins}

\usepackage[raggedright]{titlesec}
\titleformat{\section}
{\normalfont\Large\bfseries\center}{\thesection}{1em}{}
\titleformat{\subsection}
{\normalfont\large\bfseries}{\thesubsection}{1em}{}
\titleformat{\subsubsection}
{\normalfont\normalsize\bfseries}{\thesubsubsection}{1em}{}
\titleformat{\paragraph}[runin]
{\normalfont\normalsize\bfseries}{\theparagraph}{1em}{}
\titleformat{\subparagraph}[runin]
{\normalfont\normalsize\bfseries}{\thesubparagraph}{1em}{}
\titlespacing*{\section} {0pt}{3.5ex plus 1ex minus .2ex}{2.3ex plus .2ex}
\titlespacing*{\subsection} {0pt}{3.25ex plus 1ex minus .2ex}{1.5ex plus .2ex}
\titlespacing*{\subsubsection}{0pt}{3.25ex plus 1ex minus .2ex}{1.5ex plus .2ex}
\titlespacing*{\paragraph} {0pt}{3.25ex plus 1ex minus .2ex}{1em}
\titlespacing*{\subparagraph} {\parindent}{3.25ex plus 1ex minus .2ex}{1em}

\makeatletter
\renewcommand*\env@matrix[1][\arraystretch]{%
  \edef\arraystretch{#1}%
  \hskip -\arraycolsep
  \let\@ifnextchar\new@ifnextchar
  \array{*\c@MaxMatrixCols c}}
\makeatother

\usepackage{titletoc}
\contentsmargin{2em}
\dottedcontents{section}[2.7em]{}{2.8em}{1pc}
\dottedcontents{subsection}[3.7em]{}{3.4em}{1pc}
\dottedcontents{subsubsection}[7.6em]{}{3.2em}{1pc}
\dottedcontents{paragraph}[10.3em]{}{3.2em}{1pc}

\newcommand{\ra}[1]{\renewcommand{\arraystretch}{#1}}

\newcommand{\ov}{\overline}
\newcommand{\ot}{\otimes}
\newcommand{\hs}{\hspace{-0.6pt}}

\newcommand{\hf}{\hspace{-0.9pt}}

\DeclareMathOperator{\id}{id}

\DeclareMathOperator{\ide}{id}
\DeclareMathOperator{\Aut}{Aut}

\DeclareMathOperator{\Ex}{Ex}

\allowdisplaybreaks[3]

\numberwithin{table}{section}
\numberwithin{equation}{section}
\numberwithin{figure}{section}
\theoremstyle{plain}
\newtheorem{theorem}{Theorem}[section]
\theoremstyle{plain}
\newtheorem{lemma}[theorem]{Lemma}
\newtheorem{definition}[theorem]{Definition}
\newtheorem{corollary}[theorem]{Corollary}
\theoremstyle{plain}
\newtheorem{proposition}[theorem]{Proposition}
\newtheorem{notation}[theorem]{Notation}
\newtheorem{notations}[theorem]{Notations}

\newtheorem{example}[theorem]{Example}
\theoremstyle{remark}
\newtheorem{remark}[theorem]{Remark}

\theoremstyle{plain}

\newtheorem{examples}[theorem]{Examples}

\begin{document}

\title[Solutions of the braid equation and orders]{Solutions of the braid equation and orders}

\author{Jorge A. Guccione}
\address{Departamento de Matem\'atica\\ Facultad de Ciencias Exactas y Naturales-UBA, Pabell\'on~1-Ciudad Universitaria\\ Intendente Guiraldes 2160 (C1428EGA) Buenos Aires, Argentina.}
\address{Instituto de Investigaciones Matem\'aticas ``Luis A. Santal\'o"\\ Pabell\'on~1-Ciudad Universitaria\\ Intendente Guiraldes 2160 (C1428EGA) Buenos Aires, Argentina.}
\email{vander@dm.uba.ar}

\author{Juan J. Guccione}
\address{Departamento de Matem\'atica\\ Facultad de Ciencias Exactas y Naturales-UBA\\ Pabell\'on~1-Ciudad Universitaria\\ Intendente Guiraldes 2160 (C1428EGA) Buenos Aires, Argentina.}
\address{Instituto Argentino de Matem\'atica-CONICET\\ Saavedra 15 3er piso\\ (C1083ACA) Buenos Aires, Argentina.}
\email{jjgucci@dm.uba.ar}

\thanks{Jorge A. Guccione and Juan J. Guccione were supported by UBACyT 20020150100153BA (UBA) and PIP 11220110100800CO (CONICET)}

\author{Christian Valqui}
\address{Pontificia Universidad Cat\'olica del Per\'u, Secci\'on Matem\'aticas, PUCP,
Av. Universitaria 1801, San Miguel, Lima 32, Per\'u.}

\address{Instituto de Matem\'atica y Ciencias Afines (IMCA) Calle Los Bi\'ologos 245. Urb San C\'esar.
La Molina, Lima 12, Per\'u.}
\email{cvalqui@pucp.edu.pe}

\thanks{Christian Valqui was supported by PUCP-DGI- ID 329 - CAP 2016-1-0045.}


\begin{abstract}
 We introduce the notion of a non-degenerate solution of the braid equation on the incidence coalgebra of a locally finite order. Each one of these solutions induce by restriction a non-degenerate set theoretic solution over the underlying set. So, it makes sense to ask if a non-degenerate set theoretic solution on the underlying set of a locally finite order extends to a non-degenerate solution on its incidence coalgebra. In this paper we begin the study of this problem.
\end{abstract}

\maketitle

\setcounter{tocdepth}{1}
\tableofcontents

\section*{Introduction}
Let $V$ be a vector space over a field $K$. One of the more basic equations of mathematical physics is the {\em quantum Yang-Baxter equation}
$$
R_{12}\circ R_{13}\circ R_{23} = R_{23}\circ R_{13}\circ R_{12} \quad\text{in $\Aut(V\ot_K V)$}
$$
where $R\colon V\ot_K V\longrightarrow V\ot_K V$ is a bijective linear operator and $R_{ij}$ denotes $R$ acting in the $i$-th and $j$-th coordinates. Let $\tau \in \Aut(V\ot_K V)$ be the flip. Then $R$ satisfies the quantum Yang-Baxter equation if and only if $r\coloneqq \tau\circ R$ satisfies the {\em braid equation}
\begin{equation}\label{eq:trenzas}
r_{12} \circ r_{23}\circ r_{12} = r_{23} \circ r_{12}\circ r_{23}.
\end{equation}
So, both equations are equivalent and working with one or the other is a matter of taste. In the present paper we consider the second one. Since the eighties many solutions of the braid equation have been found, many of them being deformations of the flip. It is interesting to obtain solutions that are not of this type, and in \cite{Dr}, Drinfeld proposed to study the most simple of them, namely, the set theoretic ones, i.e. pairs $(X,r)$, where $X$ is a set and
$$
r\colon X\times X\longrightarrow X\times X
$$
is an invertible map satisfying~\eqref{eq:trenzas}.  Each one of these solutions yields in an evident way a linear solution on the vector space with basis~$X$. From a structural point of view this approach was considered first by Etingof, Schedler and Soloviev \cite{ESS} and Gateva-Ivanova and Van den Bergh \cite{GIVdB} for non involutive solutions, and then by Lu, Yan and Zhu \cite{LYZ} and Soloviev \cite{So} for non-degenerate not necessarily involutive solutions. In the last two decades the theory has developed rapidly, and now it is known that it has connections with bijective 1-cocycles, Bierbach groups and groups of I-type, involutive Yang-Baxter groups, Garside structures, biracks, cyclic sets, braces, Hopf algebras, matched pairs, left symmetric algebras, etcetera (see, for instance \cite{AGV}, \cite{CJO}, \cite{CJO2}, \cite{CJR}, \cite{De}, \cite{GI}, \cite{Ru}, \cite{Ta}).

\smallskip

In \cite{GGV}, the authors began the study of set-type solutions of the braid equation in the context of symmetric tensor categories. The underlying idea is simple: replace the sets by cocommutative coalgebras. The central result of that paper was the existence of the universal solutions in this setting (this generalizes the main result of \cite{LYZ}), and the main technical tool was the definition of non-degenerate maps. But this definition makes sense for non-commutative coalgebras in symmetric tensor categories, and in a forthcoming paper we will investigate the non cocommutative versions of the theoretic results established in~\cite{ESS}, \cite{LYZ} and \cite{So}. In the present paper we are interested in another type of problems, involving an important particular case, and related with the search of non-degenerate solutions on the incidence coalgebra $D$ of a locally finite poset~$(X,\le)$ (for the definitions see Subsection~\ref{subsection posets}). Each non-degenerate coalgebra automorphism
$$
r \colon D\ot_K D \longrightarrow D\ot_K D
$$
induces by restriction a non-degenerate bijection
$$
r_{|}\colon X\times X \longrightarrow X\times X.
$$
Moreover, if $r$ is a solution of the braid equation, then $r_{|}$ is a solution of the set-theoretic braid equation. So, it makes sense to study the following problems: given  a linear automorphism
$$
r\colon D\ot_K D \longrightarrow D\ot_K D
$$
such that $r_{|}$ is a non-degenerate bijection:

\begin{enumerate} [itemsep=1.0ex, topsep=1.0ex, label=\roman*)]

\item  Find necessary and sufficient conditions for $r$ to be a non-degenerate coalgebra automorphism.

\item  Assuming that $r_{|}$ is a non-degenerate solution of the set-theoretic braid equa\-tion, find necessary and sufficient conditions for $r$ to be a non-de\-gen\-er\-ate solution of the braid equation.

\end{enumerate}
In Sections~\ref{section: Properties} and~\ref{seccion construccion de automorfismos no degenerados} we solve completely the first problem (Section~\ref{seccion preliminares} is devoted to the preliminaries). The main result is Theorem~\ref{principal para no degenerados}. In Sections~\ref{section Non-degenerate solutions of the braid equation on coalgebras of orders}, \ref{section la configuracion x menor que y} and~\ref{seccion la configuracion x<y>z no flip} we consider the second problem. In Proposition~\ref{condicion para braided} we encode in a system of (non linear) equations the conditions for $r$ to be a non-degenerate solution of the braid equation. Then we analyze the meaning of the equations in Proposition~\ref{condicion para braided}, when the sum of the lengths of the involved intervals $[a,b]$, $[c,d]$ and $[e,f]$ is less than or equal to~$1$. This allows us to solve these equations in Proposition~\ref{parte lineal}, under fairly general conditions. In Theorem~\ref{Lambdas caso flip x prec y}, given $x\prec y$ in $X$,  we determine all the solutions of the equations determined by subintervals of $[x,y]$, under the hypothesis that $r$ induces the flip on $\{x,y\}\times \{x,y\}$. Using this, in Corollary~\ref{el orden x menor que y} we find all the non-degenerate solutions of the braid equation associated with the poset $(\{x,y\}, \le )$, where~$x<y$. Finally, in Section~\ref{seccion la configuracion x<y>z no flip} we give the solution of the same problem for the configuration $x\prec y \succ z$, under the hypothesis that $r_{|}$ induces a permutation on $\{x,y,z\}\times \{x,y,z\}$ that is not the flip. This allows us to obtain all the non-degenerate solutions of the braid equation associated with the poset $(\{x,y,z\}, \le )$, where~$x\prec y \succ z$, such that $r_{|}$ is not the flip.

\section{Preliminaries}\label{seccion preliminares}
In this paper we work in the category of vector spaces over a field $K$, all the maps between vector spaces are $K$-linear maps, and given vector spaces $V$ and $W$, we let $V\ot W$ denote the tensor product $V\ot_K W$ and we set $V^2\coloneqq V\otimes V$.

\subsection{Braided sets}
Let $C$ be a coalgebra.  Let $r$ be a coalgebra automorphism of $C^2$ and let
$$
\sigma\coloneqq (C\ot \epsilon)\circ r\quad\text{and} \quad \tau\coloneqq (\epsilon \ot C)\circ r.
$$
Then $r=(\sigma\ot \tau)\circ \Delta_{C^2}$. Moreover $\sigma$ and $\tau$ are the unique coalgebra morphisms with this property.

\begin{definition} A pair $(C,r)$, where $r$ is coalgebra automorphism of $C^2$, is called a {\em braided set} if $r$ satisfies the braid equation
\begin{equation}\label{ec de trenzas}
r_{12}\circ r_{23} \circ r_{12} = r_{23}\circ r_{12} \circ r_{23},
\end{equation}
where $r_{12}\coloneqq r\ot C$ and $r_{23}\coloneqq C \ot r$, and it is called {\em non-degenerate} if there exist maps $\ov \sigma\colon C^2 \to C$ and $\ov \tau\colon C^2 \to C$ such that
\begin{align}
&\ov \sigma\circ (C\ot \sigma)\circ (\Delta\ot C)= \sigma\circ (C\ot \ov \sigma)\circ (\Delta\ot C)= \epsilon \ot C\label{no deg a izq}
\shortintertext{and}
&\ov \tau\circ (\tau \ot C)\circ (C \ot \Delta)= \tau \circ (\ov \tau \ot C)\circ (C \ot \Delta)= C \ot \epsilon.\label{no deg a der}
\end{align}
If $(C,r)$ is a non-degenerate pair, then we say that $r$ is non-degenerate.
\end{definition}

A direct computation shows that $r$ is non-degenerate if and only if the maps $(C \ot \sigma)\circ (\Delta \ot C)$ and $(\tau \ot C)\circ (C\ot \Delta)$ are isomorphisms. Moreover, their compositional inverses are the maps $(C \ot \ov \sigma)\circ (\Delta \ot C)$ and $(\ov \tau \ot C)\circ (C\ot \Delta)$, respectively.
This implies  that $\ov \sigma$ and $\ov \tau$ are coalgebra morphisms.

\subsection{Posets}\label{subsection posets}
A {\em partially ordered set} or {\em poset} is a pair $(X,\le)$ consisting of a set $X$ endowed with a binary relation~$\le$, called {\em an order}, that is reflexive, antisymmetric and transitive. For the sake of brevity from now on we will say that $X$ is a poset, without explicit mention of the order. As usual, for $a,b\in X$ we write $a<b$ to mean that $a\le b$ and~$a\ne b$. Two elements $a,b$ of $X$ are {\em comparable} if $a\le b$ or $b\le a$. Otherwise they are {\em incomparable}. A poset $X$ is a {\em totally ordered set} if each pair of elements of~$X$ is comparable. A {\em connected component} of $X$ is an equivalence class of the
 equivalence relation generated by the relation $x\sim y$ if $x$ and $y$ are comparable. Let $X$ be a poset. Each subset $Y$ of $X$  becomes a poset simply by restricting the order relation of $X$ to $Y^2$. A subset $Y$ of $X$ is a chain of $X$ if it is a totally ordered set. The {\em height} of a finite chain $a_0<\cdots < a_n$ is $n$. The {\em height} $h(X)$ of a finite poset $X$ is the height of its largest chain. Let $a,b\in X$. The {closed interval} $[a,b]$ is the set of all the elements $c$ of $X$ such that $a\le c \le b$. We say that~$b$ covers $a$, and we write $a\prec b$ (or $b\succ a$), if $[a,b]=\{a,b\}$. A poset $X$ is {\em locally finite} if $[a,b]$ is finite for all $a,b\in X$.

\subsubsection{The incidence coalgebra of a locally finite poset}
Let $(X,\le)$ be a locally finite poset. Set $Y\coloneqq\{(a,b)\in X\times X: a\le b\}$. It is well known that $D\coloneqq KY$ is a coalgebra, called the {\em incidence coalgebra of $X$}, via
$$
\Delta(a,b)\coloneqq \sum_{c\in [a,b]} (a,c)\ot (c,b).
$$
Consider $KX$ endowed with the coalgebra structure determined by requiring that each $x\in X$ is a group--like element. The $K$-linear map $\iota \colon KX \to D$ defined by $\iota(x)\coloneqq (x,x)$ is an injective coalgebra morphism, whose image is the subcoalgebra of $D$ spanned by the group--like elements of $KY$.

Let $r\colon D\ot D \longrightarrow D\ot D$ be a $K$-linear map and let
$$
\bigl(\lambda_{a|b|c|d}^{e|f|g|h}\bigr)_{(a,b),(c,d),(e,f),(g,h)\in Y}
$$
be the family of scalars defined by
\begin{equation}\label{def de los lambdas}
r((a,b)\ot (c,d))= \sum_{\substack{e\le f\\ g\le h}} \lambda_{a|b|c|d}^{e|f|g|h} (e,f)\ot (g,h).
\end{equation}
From now on we assume that $r$ is invertible.

\begin{remark} \label{compatibilidad con epsilon}
Let $T:= (a,b)\ot (c,d)$. Since
\begin{align*}
& (\epsilon\ot \epsilon)(T) = \delta_{ab}\delta_{cd},\qquad (\epsilon\ot \epsilon)\circ r(T)=\sum_{e,g} \lambda_{a|b|c|d}^{e|e|g|g},\\
& \Delta_{D^2}\circ r (T) = \sum_{\substack{e\le f\\ g\le h}}  \sum_{\substack{y\in [e,f]\\ z\in [g,h]}} \lambda_{a|b|c|d}^{e|f|g|h} (e,y)\ot (g,z)\ot (y,f)\ot (z,h)
\shortintertext{and}
& (r \ot r)\circ \Delta_{D^2} (T) = \sum_{\substack{p\in [a,b]\\ q \in [c,d]}} \sum_{\substack{e\le y\\ g\le z}} \sum_{\substack{y'\le f\\ z'\le h}} \lambda_{a|p|c|q}^{e|y|g|z} \lambda_{p|b|q|d}^{y'|f|z'|h} (e,y)\ot (g,z)\ot (y',f)\ot (z',h),
\end{align*}
the map $r$ is a coalgebra automorphism if and only if the following facts hold:

\begin{itemize}[itemsep=1.0ex, topsep=1.0ex]

\item[-] for $a\le b$ and $c\le d$,
\begin{equation} \label{ordenes: r preserva la counidad}
\qquad\quad\sum_{e,g} \lambda_{a|b|c|d}^{e|e|g|g}=\delta_{ab}\delta_{cd},
\end{equation}

\item[-] for each $a\le b$, $c\le d$, $e\le f$ and $g\le h$,
\begin{equation} \label{eq1}
 \sum_{(p,q)\in [a,b] \times[c,d]} \lambda_{a|p|c|q}^{e|y|g|z} \lambda_{p|b|q|d}^{y|f|z|h} = \lambda_{a|b|c|d}^{e|f|g|h}
\end{equation}
 for all $y$ and $z$ such that $e\le y\le f$ and $g\le z\le h$,

\item[-] For each $a\le b$, $c\le d$, $e\le f$ and $g\le h$,
\begin{equation} \label{eq2}
 \sum_{(p,q)\in [a,b] \times[c,d]} \lambda_{a|p|c|q}^{e|y|g|z} \lambda_{p|b|q|d}^{y'|f|z'|h}=0,
\end{equation}
for all  $y'\le f$, $e\le y$, $z'\le h$ and $g\le z$ such that $(y,z)\ne (y',z')$.
\end{itemize}
\end{remark}

\begin{remark}\label{formulas para sigma y tau}
By the very definitions of $\sigma$ and $\tau$, it is clear that
$$
\sigma((a,b)\ot (c,d))= \sum_{\substack{e\le f\\ g}} \lambda_{a|b|c|d}^{e|f|g|g} (e,f)\quad\text{and}\quad \tau((a,b)\ot (c,d))= \sum_{\substack{e\\ g\le h}} \lambda_{a|b|c|d}^{e|e|g|h} (g,h).
$$
\end{remark}

\section{Non-degenerate automorphisms of the incidence coalgebra}\label{section: Properties}
In this section we determine the main properties of the coefficients $\lambda_{a|b|c|d}^{e|f|g|h}$ and the maps ${}^a \hf (-)$ and $(-)\hs^c$ determined by a non-degenerate coalgebra automorphism of~$D\ot D$ (for the definition of these maps see Notation~\ref{coordenadas conjuntistas}).

\begin{remark}\label{ordenes: r sobre elementos de tipo grupo}
Let $r\colon D\ot D \longrightarrow D\ot  D$ be a non-degenerate coalgebra automorphism. Since $r$ maps group--like elements to group--like elements, for all $a,c\in X$ there exist $e,g\in X$ such that
$$
\lambda_{a|a|c|c}^{e|e|g|g}=1\quad\text{and}\quad \lambda_{a|a|c|c}^{e'|e''|g'|g''}=0\quad\text{for all $(e',e'',g',g'') \ne (e,e,g,g)$.}
$$
So, $r$ induces a map $r_{|}\colon X\times X\longrightarrow X\times X$.  The same argument applied to $r^{-1}$ shows that $r_{|}$ is a bijection. Moreover, since $\ov \sigma$ and $\ov \tau$ map group--like elements to group--like elements, $r_{|}$ is non-degenerate.
\end{remark}

\begin{notation}\label{coordenadas conjuntistas}
In the sequel if $r_{|}(a,c)=(e,g)$ we write ${}^a \hf c\coloneqq e$ and  $a\hs^c \coloneqq g$.
\end{notation}

For the rest of the section we fix a non-degenerate coalgebra automorphism $r$ of $D\ot  D$ and we determine properties of the coefficients $\lambda_{a|b|c|d}^{e|f|g|h}$ and the maps ${}^a \hf (-)$ and $(-)\hs^c$. Note that $r_{|}$ being non-degenerate means that
the maps ${}^a \hf (-)$ and $(-)\hs^c$ are bijections.

\begin{proposition}\label{iguales tensor cover}
Let $a,c,d\in X$. If $d$ covers $c$, then $a\hs^c = a\hs^d$, ${}^a\!d$ covers ${}^a\!c$ and there exists $\alpha\in K^{\times}$ and $\beta\in K$ such that
$$
r((a,a)\ot (c,d)) = \alpha   ({}^a\!c,{}^a\!d)\ot (a\hs^c,a\hs^c)  + \beta ({}^a\!c,{}^a\!c)\ot (a\hs^c,a\hs^c) - \beta ({}^a\!d,{}^a\!d)\ot (a\hs^c,a\hs^c).
$$
\end{proposition}

\begin{proof}
Under the hypothesis, equality~\eqref{eq1} says that for each $e\le y \le f$ and each $g\le z\le h$,
\begin{equation}\label{eq3}
\lambda_{a|a|c|d}^{e|f|g|h} = \lambda_{a|a|c|c}^{e|y|g|z} \lambda_{a|a|c|d}^{y|f|z|h} + \lambda_{a|a|c|d}^{e|y|g|z} \lambda_{a|a|d|d}^{y|f|z|h}.
\end{equation}
If $e<f$ and $g<h$, then taking $y=f$ and $z=g$, we obtain that $\lambda_{a|a|c|d}^{e|f|g|h}=0$. A similar argument proves that if $e=f$ and there exists $z$ such that $g < z <h$, then also $\lambda_{a|a|c|d}^{e|f|g|h}=0$. Furthermore, by symmetry the same occurs if $g=h$ and there exists~$y$ such that $e<y<f$. So, if $\lambda_{a|a|c|d}^{e|f|g|h}\ne 0$, then we have the following possibilities:
\begin{gather*}
 \text{a)}\ e=f\text{ and } h \text{ covers } g \qquad  \text{b)}\ g=h\text{ and } f\text{ covers } e \qquad  \text{c)}\ e=f \text{ and } h=g
\end{gather*}
Next we consider each of these cases separately:

\begin{enumerate}[itemsep=1.0ex, topsep=1.0ex, label=\alph*)]

\item Taking $y=e$ and $z=g$ in equality~\eqref{eq3}, we obtain that
$$
\quad\qquad  \lambda_{a|a|c|d}^{e|e|g|h}=\lambda_{a|a|c|c}^{e|e|g|g} \lambda_{a|a|c|d}^{e|e|g|h} + \lambda_{a|a|c|d}^{e|e|g|g} \lambda_{a|a|d|d}^{e|e|g|h} = \lambda_{a|a|c|c}^{e|e|g|g} \lambda_{a|a|c|d}^{e|e|g|h},
$$
while taking $y=e$ and $z=h$ in equality~\eqref{eq3}, we obtain that
$$
\quad\qquad  \lambda_{a|a|c|d}^{e|e|g|h}=\lambda_{a|a|c|c}^{e|e|g|h} \lambda_{a|a|c|d}^{e|e|h|h} + \lambda_{a|a|c|d}^{e|e|g|h} \lambda_{a|a|d|d}^{e|e|h|h} = \lambda_{a|a|c|d}^{e|e|g|h} \lambda_{a|a|d|d}^{e|e|h|h}.
$$
Therefore, if $ \lambda_{a|a|c|d}^{e|e|g|h}\ne 0$, then $e={}^a\!c={}^a\!d$, which is impossible, since
${}^a\!(-)$ is a bijection.

\item Taking $y=e$ and $z=g$ in equality~\eqref{eq3}, we obtain that
$$
\quad\qquad  \lambda_{a|a|c|d}^{e|f|g|g}=\lambda_{a|a|c|c}^{e|e|g|g} \lambda_{a|a|c|d}^{e|f|g|g} + \lambda_{a|a|c|d}^{e|e|g|g} \lambda_{a|a|d|d}^{e|f|g|g} = \lambda_{a|a|c|c}^{e|e|g|g} \lambda_{a|a|c|d}^{e|f|g|g},
$$
while taking $y=f$ and $z=g$ in equality~\eqref{eq3}, we obtain that
$$
\quad\qquad  \lambda_{a|a|c|d}^{e|f|g|g}=\lambda_{a|a|c|c}^{e|f|g|g} \lambda_{a|a|c|d}^{f|f|g|g} + \lambda_{a|a|c|d}^{e|f|g|g} \lambda_{a|a|d|d}^{f|f|g|g} = \lambda_{a|a|c|d}^{e|f|g|g} \lambda_{a|a|d|d}^{f|f|g|g}.
$$
Therefore, if $\lambda_{a|a|c|d}^{e|f|g|g}\ne 0$ then $(f,g)=({}^a\!d,a\hs^d)$ and $(e,g)=({}^a\!c,a\hs^c)$.

\item Taking $y=e$ and $z=g$ in equality~\eqref{eq3}, we obtain that
$$
\quad\qquad  \lambda_{a|a|c|d}^{e|e|g|g}=\lambda_{a|a|c|c}^{e|e|g|g} \lambda_{a|a|c|d}^{e|e|g|g} + \lambda_{a|a|c|d}^{e|e|g|g} \lambda_{a|a|d|d}^{e|e|g|g} = (\lambda_{a|a|c|c}^{e|e|g|g} +\lambda_{a|a|d|d}^{e|e|g|g}) \lambda_{a|a|c|d}^{e|e|g|g},
$$
which implies that $(e,g)=({}^a\!c,a\hs^c)$ or $(e,g)=({}^a\!d,a\hs^d)$, when $\lambda_{a|a|c|d}^{e|e|g|g}\ne 0$.
\end{enumerate}
Thus,
\begin{align*}
r((a,a)\ot (c,d)) &=  \lambda_{a|a|c|d}^{{}^a\!c | {}^a\!d | a\hs^c | a\hs^c} ({}^a\!c,{}^a\!d)\ot (a\hs^c,a\hs^c)
 + \lambda_{a|a|c|d}^{{}^a\!c | {}^a\!c | a\hs^c | a\hs^c} ({}^a\!c,{}^a\!c)\ot (a\hs^c,a\hs^c)\\
& + \lambda_{a|a|c|d}^{{}^a\!d | {}^a\!d | a\hs^d | a\hs^d} ({}^a\!d,{}^a\!d)\ot (a\hs^d,a\hs^d).
\end{align*}
Since $r$ and $r_{|}$ are bijective,  $\alpha\coloneqq \lambda_{a|a|c|d}^{{}^a\!c | {}^a\!d | a\hs^c | a\hs^c} \ne 0$, and then $a\hs^c = a\hs^d$ and ${}^a\!d$ covers ${}^a\!c$. Also notice that
$$
\beta\coloneqq \lambda_{a|a|c|d}^{{}^a\!c | {}^a\!c | a\hs^c | a\hs^c} = - \lambda_{a|a|c|d}^{{}^a\!d | {}^a\!d | a\hs^d | a\hs^d},
$$
because $(\epsilon \ot \epsilon)\circ r((a,a)\ot (c,d))=0$.
\end{proof}

\begin{proposition}\label{covers tensor iguales}
Let $a,b,c\in X$. If $b$ covers $a$, then $b\hs^c$ covers $a\hs^c$, ${}^a\!c={}^b\!c$ and there exists $\alpha\in K^{\times}$ and $\beta\in K$ such that
$$
r((a,b)\ot (c,c)) = \alpha   ({}^a\!c,{}^a\!c)\ot (a\hs^c,b\hs^c)  + \beta ({}^a\!c,{}^a\!c)\ot (a\hs^c,a\hs^c) - \beta ({}^a\!c,{}^a\!c)\ot (b\hs^c,b\hs^c).
$$
\end{proposition}

\begin{proof}
Apply Proposition~\ref{iguales tensor cover} to $\tau\circ r  \circ \tau$, where $\tau$ is the flip.
\end{proof}

\begin{corollary} \label{son iso de ordenes}
For each $a\in X$ the maps ${}^a\!(-)$ and $(-)\hs^a$ are automorphisms of orders. Moreover, if $a$ and $b$ are comparable or, more generally, if $a$ and $b$ belong to the same component of $X$, then ${}^a\!(-)={}^b\!(-)$ and $(-)\hs^a=(-)\hs^b$.
\end{corollary}

\begin{notations}
We let  ${}^{\bar{a}}\!(-)$ and $(-)\hs^{\bar{a}}$ denote the inverse maps of ${}^a\!(-)$ and $(-)\hs^a$, respectively. Note that here $\bar{a}$ is not an element of $X$.
\end{notations}

\begin{lemma}\label{caso a=b}
Let $a,c, d, e, f, g,h\in X$ such that $c\le d$, $e\le f$ and $g\le h$. If $g\ne a\hs^c$  or $h\ne a\hs^c$, then $\lambda_{a|a|c|d}^{e|f|g|h}= 0$.
\end{lemma}

\begin{proof}
By Remark~\ref{ordenes: r sobre elementos de tipo grupo} and Proposition~\ref{iguales tensor cover} we know that the statement is true when $\mathfrak{h}[c,d]\le 1$. Assume that it is true when $\mathfrak{h}[c,d]\le n\ge 1$ and that $\mathfrak{h}[c,d]= n+1$. If $g\ne a\hs^c$, then
$$
\lambda_{a|a|c|d}^{e|f|g|h} = \sum_{q \in [c,d]} \lambda_{a|a|c|q}^{e|e|g|g} \lambda_{a|a|q|d}^{e|f|g|h}=0,
$$
because $\lambda_{a|a|d|d}^{y|f|g|h}=0$ since $g\ne a\hs^d=a\hs^c$ and  $\lambda_{a|a|c|q}^{e|e|g|g}=0$ when $q<d$, by the inductive hypothesis;
while if $h\ne a\hs^c$, then
$$
\lambda_{a|a|c|d}^{e|f|g|h} = \sum_{q \in [c,d]} \lambda_{a|a|c|q}^{e|e|g|h} \lambda_{a|a|q|d}^{e|f|h|h}=0,
$$
because $\lambda_{a|a|c|c}^{e|e|g|h}=0$ since $h\ne a\hs^c$ and $\lambda_{a|a|q|d}^{e|f|h|h}=0$ when $q>c$, by the inductive hypothesis, since $a\hs^q = a\hs^c \ne h$.
\end{proof}

\begin{proposition}\label{donde puede no anularse}
Let $a\le b$, $c\le d$, $e\le f$ and $g\le h$. If $\lambda_{a|b|c|d}^{e|f|g|h}\ne 0$, then it is true that $a\hs^c\le g\le h\le b\hs^c$ and ${}^a\!c\le e \le f \le {}^a\!d$.
\end{proposition}

\begin{proof}
By symmetry it is sufficient to prove that if $a\hs^c\le g\le h\le b\hs^c$ is false, then $\lambda_{a|b|c|d}^{e|f|g|h}= 0$. By Lemma~\ref{caso a=b} this is true when $a=b$. Assume that it is true when $\mathfrak{h}[a,b]=m$ and that $\mathfrak{h}[a,b]=m+1$. On one hand, if $a\hs^c \nleq g$, then
\begin{align*}
\lambda_{a|b|c|d}^{e|f|g|h} & = \sum_{(p,q)\in [a,b]\times [c,d]} \lambda_{a|p|c|q}^{e|e|g|g} \lambda_{p|b|q|d}^{e|f|g|h}\\
&= \sum_{_{(p,q)\in [a,b)\times [c,d]}} \lambda_{a|p|c|q}^{e|e|g|g} \lambda_{p|b|q|d}^{e|f|g|h} + \sum_{q \in [c,d]} \lambda_{a|b|c|q}^{e|e|g|g} \lambda_{b|b|q|d}^{e|f|g|h}\\
&=0
\end{align*}
because

\begin{itemize}[itemsep=1.0ex, topsep=1.0ex]

\item[-] $\lambda_{b|b|q|d}^{e|f|g|h}=0$ by Lemma~\ref{caso a=b}, since by Corollary~\ref{son iso de ordenes}, we have $b\hs^q=b\hs^c \nleq g$,

\item[-]  $\lambda_{a|p|c|q}^{e|e|g|g}=0$ when $p<b$, by the inductive hypothesis.

\end{itemize}
On the other hand, if $h \nleq b\hs^c$, then
\begin{align*}
\lambda_{a|b|c|d}^{e|f|g|h} & = \sum_{(p,q)\in [a,b]\times [c,d]} \lambda_{a|p|c|q}^{e|e|g|h} \lambda_{p|b|q|d}^{e|f|h|h}\\
&= \sum_{(p,q)\in (a,b]\times [c,d]} \lambda_{a|p|c|q}^{e|e|g|h} \lambda_{p|b|q|d}^{e|f|h|h} + \sum_{q\in [c,d]} \lambda_{a|a|c|q}^{e|e|g|h} \lambda_{a|b|q|d}^{e|f|h|h}\\
&=0
\end{align*}
because

\begin{itemize}[itemsep=1.0ex, topsep=1.0ex]

\item[-] $\lambda_{p|b|q|d}^{e|f|h|h}=0$ when $p>a$, by the inductive hypothesis, since by Corollary~\ref{son iso de ordenes}, we have $h \nleq b\hs^q=b\hs^c$,

\item[-] $\lambda_{a|a|c|q}^{e|e|g|h}=0$ by Lemma~\ref{caso a=b}, since by Corollary~\ref{son iso de ordenes}, we have $h\nleq a\hs^c$.

\end{itemize}
This finishes the proof.
\end{proof}

\begin{corollary}
The following formula holds:
$$
r((a,b)\ot (c,d)) =\sum_{\substack{\{(x,y):a\le x\le y \le b\}\\ \{(w,z):c\le w\le z \le d\}}} \lambda_{a|b|c|d}^{{}^a\!w | {}^a\!z | x\hs^c | y\hs^c} ({}^a\!w, {}^a\!z)\ot (x\hs^c, y\hs^c).
$$
\end{corollary}

\begin{proof}
It follows immediately from Proposition~\ref{donde puede no anularse} and Corolary~\ref{son iso de ordenes}.
\end{proof}

\begin{proposition}\label{lambdas como productos}
Let $a\le b$, $c\le d$ $e\le f$ and $g\le h$ such that $a\hs^c\le g\le h\le b\hs^c$ and ${}^a\!c\le e \le f \le {}^a\!d$.  For each $y,z\in X$ such that $e\le y \le f$ and $g\le z \le h$, the following equality holds:
\begin{equation}\label{producto}
\lambda^{e|f|g|h}_{a|b|c|d} = \lambda_{a | z\hs^{\bar{c}} | c | {}^{\bar{a}}\!y}^{e|y|g|z} \lambda_{z\hs^{\bar{c}}| b | {}^{\bar{a}}\!y | d}^{y|f|z|h}.
\end{equation}
\end{proposition}

\begin{proof}
By Proposition~\ref{donde puede no anularse} and Corollary~\ref{son iso de ordenes}, if $\lambda_{a|p|c|q}^{e|y|g|z} \lambda_{p|b|q|d}^{y|f|z|h}\ne 0$, then
$$
{}^a\!c \le e \le y \le {}^a\!q,\quad {}^a\!q \le y \le f \le   {}^a\!d, \quad a\hs^c \le g \le z \le p\hs^c\quad \text{and}\quad p\hs^c \le z \le h \le b\hs^c,
$$
So, $q={}^{\bar{a}}\!y$, $p=z\hs^{\bar{c}}$, and the result follows from equality~\eqref{eq1}.
\end{proof}

\begin{examples}\label{ejemplos de lambdas como productos} Let $a\le b$ and $c\le d$. From the previous proposition it follows that:

\begin{enumerate}[itemsep=1.0ex, topsep=1.0ex, label=\emph{\arabic*)}]

\item  for each $e,g\in X$ such that $a\hs^c\le g\le b\hs^c$ and ${}^a\!c\le e  \le {}^a\!d$,
$$
\lambda_{a|b|c|d}^{e|e|g|g}= \lambda_{a | g\hs^{\bar{c}} | c | {}^{\bar{a}}\!e}^{e|e|g|g} \lambda_{g\hs^{\bar{c}} | b | {}^{\bar{a}}\!e | d}^{e|e|g|g},
$$

\item for each $e,f,g\in X$ such that $a\hs^c\le g\le b\hs^c$ and ${}^a\!c\le e \prec f \le {}^a\!d$,
$$
\lambda^{e|f|g|g}_{a|b|c|d} = \lambda_{a | g\hs^{\bar{c}} | c | {}^{\bar{a}}\!e}^{e|e|g|g} \lambda_{g\hs^{\bar{c}} | b | {}^{\bar{a}}\!e | d}^{e|f|g|g} = \lambda_{a | g\hs^{\bar{c}} | c | {}^{\bar{a}}\!f}^{e|f|g|g} \lambda_{g\hs^{\bar{c}} | b {}^{\bar{a}}\!f | d}^{f|f|g|g},
$$

\item for each $e,g,h\in X$ such that ${}^a\!c\le e\le {}^a\!d$ and $a\hs^c\le g \prec h\le b\hs^c$,
$$
\lambda^{e|e|g|h}_{a|b|c|d} = \lambda_{a | g\hs^{\bar{c}} | c | {}^{\bar{a}}\!e}^{e|e|g|g} \lambda_{g\hs^{\bar{c}} | b | {}^{\bar{a}}\!e | d}^{e|e|g|h} =  \lambda_{a | h\hs^{\bar{c}} | c | {}^{\bar{a}}\!e}^{e|e|g|h} \lambda_{h\hs^{\bar{c}} | b | {}^{\bar{a}}\!e | d}^{e|e|h|h}.
$$
\end{enumerate}
\end{examples}

\begin{notations}
For $p,q,m,n\in \mathds{N}_0$, we let ${}_r\Lambda_{pq}^{mn}$ denote the restriction of the family
$$
\bigl(\lambda_{a|b|c|d}^{e|f|g|h}\bigr)_{(a,b),(c,d),(e,f),(g,h)\in Y}
$$
to the set of indices $\bigl\{\bigl((a,b),(c,d),(e,f),(g,h)\bigr)\bigr\}$ such that
$$
\text{$\mathfrak{h}[a,b]=p$, $\mathfrak{h}[c,d]=q$, $\mathfrak{h}[e,f]=m$, $\mathfrak{h}[g,h]=n$, $[g,h]\subseteq [a\hs^c,b\hs^c]$ and $[e,f]\subseteq [{}^a\!c,{}^a\!d]$}.
$$
Moreover, we set
$$
{}_r\Lambda_u^v \coloneqq \bigcup_{\substack{p+q=u\\ m+n=v}} {}_r\Lambda_{pq}^{mn},
$$
and we denote by $\Ex({}_r\Lambda_n^0)$ the restriction of ${}_r\Lambda_n^0$ to the set
$$
\bigl\{\bigl((a,b),(c,d),(e,e),(g,g)\bigr)\bigr): \text{$(e,g)= ({}^a\!c, a\hs^{c})$ or $(e,g)= ({}^b\!d, b\hs^d)$} \bigr\}.
$$
Note that $\Ex({}_r\Lambda_0^0)= {}_r\Lambda_0^0$ and $\Ex({}_r\Lambda_1^0)= {}_r\Lambda_1^0$.
\end{notations}

\begin{proposition} Let $\tilde{r}\colon D\ot D \longrightarrow D\ot  D$ be a non-degenerate coalgebra automorphism. Then $\tilde{r}= r$ if and only if
\begin{equation}\label{eq: unicidad}
\Ex({}_{\tilde{r}}\Lambda_n^0)=\Ex({}_r\Lambda_n^0)\quad\text{for all $n$} \qquad\text{and}\qquad {}_{\tilde{r}}\Lambda_1^1 = {}_r\Lambda_1^1.
\end{equation}
\end{proposition}

\begin{proof} Clearly the conditions are necessary. So, we only need to prove that they are sufficient.
For the sake of brevity we write
$$
\widetilde{\lambda}_{a|b|c|d}^{e|f|g|h}\coloneqq {}_{\tilde{r}}\Lambda \bigl((a,b),(c,d),(e,f),(g,h)\bigr).
$$
For each $e,g\in X$ such that $a\hs^c\le g\le b\hs^c$ and ${}^a\!c\le e  \le {}^a\!d$, we have
$$
\lambda_{a|b|c|d}^{e|e|g|g}= \lambda_{a | g\hs^{\bar{c}} | c | {}^{\bar{a}}\!e}^{e|e|g|g} \lambda_{g\hs^{\bar{c}} | b | {}^{\bar{a}}\!e | d}^{e|e|g|g} = \widetilde{\lambda}_{a | g\hs^{\bar{c}} | c | {}^{\bar{a}}\!e}^{e|e|g|g} \widetilde{\lambda}_{g\hs^{\bar{c}} | b | {}^{\bar{a}}\!e | d}^{e|e|g|g} = \widetilde{\lambda}_{a|b|c|d}^{e|e|g|g},
$$
where the first and the last equality hold by  item~1) of Examples~\ref{ejemplos de lambdas como productos}, and the second equality holds since
$$
\lambda_{a | g\hs^{\bar{c}} | c | {}^{\bar{a}}\!e}^{e|e|g|g}, \lambda_{g\hs^{\bar{c}} | b | {}^{\bar{a}}\!e | d}^{e|e|g|g}\in \bigcup_{u\in \mathds{N}_0} \Ex({}_r\Lambda_u^0).
$$
So, ${}_r\Lambda_u^0 = {}_{\tilde{r}}\Lambda_u^0$ for all~$u$.

Next we prove by induction on $u$ that ${}_r\Lambda_u^1 = {}_{\tilde{r}}\Lambda_u^1$ for all~$u$. For $u=0$ this is true, since
${}_r\Lambda_0^1 = {}_{\tilde{r}}\Lambda_0^1=\emptyset$, and for $u=1$ this is true by hypothesis. Take $u>1$ and assume that
${}_r\Lambda_{u'}^1 = {}_{\tilde{r}}\Lambda_{u'}^1$ for all $u'<u$. Consider $\lambda_{a|b|c|d}^{e|f|g|h}$ with
$$
\text{$\mathfrak{h}[a,b]+\mathfrak{h}[c,d]=u$, $\mathfrak{h}[e,f]+\mathfrak{h}[g,h]=1$, $a\hs^c\le g\le h\le b\hs^c$ and ${}^a\!c\le e \le f \le {}^a\!d$.}
$$
Assume first that $g=h$. If $g\hs^{\bar{c}}>a$ or ${}^{\bar{a}}\! e>c$, then taking $y=e$ in~\eqref{producto}, we obtain
$$
\lambda_{a|b|c|d}^{e|f|g|g}= \lambda_{a | g\hs^{\bar{c}} | c | {}^{\bar{a}}\! e}^{e|e|g|g} \lambda_{g\hs^{\bar{c}} | b | {}^{\bar{a}}\!e | d}^{e|f|g|g} = \widetilde{\lambda}_{a | g\hs^{\bar{c}} | c | {}^{\bar{a}}\! e}^{e|e|g|g} \widetilde{\lambda}_{g\hs^{\bar{c}} | b | {}^{\bar{a}}\!e | d}^{e|f|g|g}= \widetilde{\lambda}_{a|b|c|d}^{e|f|g|g},
$$
since $\mathfrak{h}[g\hs^{\bar{c}}, b] + \mathfrak{h}[{}^{\bar{a}}\!e,d]< \mathfrak{h}[a,b] + \mathfrak{h}[c,d]=u$. Else, taking $y=f$, we obtain
$$
\lambda_{a|b|c|d}^{e|f|g|g}= \lambda_{a | g\hs^{\bar{c}} | c | {}^{\bar{a}}\! f}^{e|f|g|g} \lambda_{g\hs^{\bar{c}} | b | {}^{\bar{a}}\!f | d}^{f|f|g|g} = \widetilde{\lambda}_{a | g\hs^{\bar{c}} | c | {}^{\bar{a}}\! f}^{e|f|g|g} \widetilde{\lambda}_{g\hs^{\bar{c}} | b | {}^{\bar{a}}\!f | d}^{f|f|g|g}= \widetilde{\lambda}_{a|b|c|d}^{e|f|g|g},
$$
because $b> g\hs^{\bar{c}}$ or $d> {}^{\bar{a}}\!f$, since otherwise $u= \mathfrak{h}[a,b]+\mathfrak{h}[c,d]=\mathfrak{h}[a,a]+\mathfrak{h}[e,f]=1$.

A similar argument yields $\lambda_{a|b|c|d}^{e|e|g|h}=\widetilde{\lambda}_{a|b|c|d}^{e|e|g|h}$ for $g \prec h$ and concludes the proof that
${}_r\Lambda_u^1 = {}_{\tilde{r}}\Lambda_u^1$.

Finally we prove using induction on $v$, that ${}_r\Lambda_u^v = {}_{\tilde{r}}\Lambda_u^v$ for all $v>1$ and all $u$. Take $v>1$ and assume that
${}_r\Lambda_{u}^{v'} = {}_{\tilde{r}}\Lambda_{u}^{v'}$ for all $v'<v$. Consider
$$
\lambda_{a|b|c|d}^{e|f|g|h} \quad\text{with $\mathfrak{h}[e,f]+\mathfrak{h}[g,h]=v$, $a\hs^c\le g\le h\le b\hs^c$ and ${}^a\!c\le e \le f \le {}^a\!d$.}
$$
If $e<f$ and $g<h$, then we take $y=f$ and $z=g$ in~\eqref{producto} and we obtain
$$
\lambda_{a|b|c|d}^{e|f|g|h}= \lambda_{a | g\hs^{\bar{c}} | c | {}^{\bar{a}}\! f}^{e|f|g|g} \lambda_{g\hs^{\bar{c}} | b | {}^{\bar{a}}\!f | d}^{f|f|g|h} = \widetilde{\lambda}_{a | g\hs^{\bar{c}} | c | {}^{\bar{a}}\! f}^{e|f|g|g} \widetilde{\lambda}_{g\hs^{\bar{c}} | b | {}^{\bar{a}}\!f | d}^{f|f|g|h}= \widetilde{\lambda}_{a|b|c|d}^{e|f|g|h},
$$
since $\mathfrak{h}[e,f]+\mathfrak{h}[g,g]<v$ and $\mathfrak{h}[f,f]+\mathfrak{h}[g,h]<v$.

Else $g=h$ or $e=f$. In the first case there exists $y$ with $e<y<f$, and so, by~\eqref{producto} we obtain
$$
\lambda_{a|b|c|d}^{e|f|g|g}= \lambda_{a | g\hs^{\bar{c}} | c | {}^{\bar{a}}\! y}^{e|y|g|g} \lambda_{g\hs^{\bar{c}} | b | {}^{\bar{a}}\!y | d}^{y|f|g|g} = \widetilde{\lambda}_{a | g\hs^{\bar{c}} | c | {}^{\bar{a}}\!y}^{e|y|g|g} \widetilde{\lambda}_{g\hs^{\bar{c}} | b | {}^{\bar{a}}\!y | d}^{y|f|g|g}= \widetilde{\lambda}_{a|b|c|d}^{e|f|g|g},
$$
since $\mathfrak{h}[e,y]+\mathfrak{h}[g,g]<v$ and $\mathfrak{h}[y,f]+\mathfrak{h}[g,g]<v$.
A similar argument proves the case $e=f$. By Proposition~\ref{donde puede no anularse} this concludes the proof.
\end{proof}

\begin{definition}
Let $e\le f$ and $g\le h$, and let
$$
e=y_0<\dots < y_j=f\quad\text{and}\quad g=z_0<\dots < z_k=h
$$
be maximal chains. A configuration for the two given chains is a family $(a_i)_{i=0,\dots,j+k}$ with $a_i=(\alpha_i,\beta_i)\in \mathds{N}_0^2$ such that $a_0=(0,0)$, $a_{j+k}=(j,k)$, $\alpha_i\le \alpha_{i+1}$, $\beta_i\le \beta_{i+1}$ and $\alpha_{i+1}- \alpha_{i}+\beta_{i+1}- \beta_{i}=1$ for $i=0,\dots,j+k-1$.
\end{definition}

\begin{proposition}\label{factorizacion}
Let $e\le f$, $g\le h$, $a\le b$ and $c\le d$, such that $a\hs^c\le g\le h\le b\hs^c$ and ${}^a\!c\le e \le f \le {}^a\!d$. Let $e=y_0<\dots < y_j=f$ and $g=z_0<z_1<\dots < z_k=h$ be maximal chains and let $(a_i)_{i=0,\dots,j+k}$ be a configuration for the two given chains. Then
$$
\lambda_{a|b|c|d}^{e|f|g|h}=\lambda_{a | g\hs^{\bar{c}} | c | {}^{\bar{a}}\! e}^{e|e|g|g} \lambda_{h\hs^{\bar{c}} | b | {}^{\bar{a}}\!f | d}^{f|f|h|h}\prod_{i=1}^{k+j}\lambda_i,\qquad \text{where}\quad \lambda_i= \lambda_{z_{\beta_{i-1}}\hs^{\bar{c}} | z_{\beta_{i}}\hs^{\bar{c}} | {}^{\bar{a}}\!y_{\alpha_{i-1}} | {}^{\bar{a}}\!y_{\alpha_{i}}}^{y_{\alpha_{i-1}}|y_{\alpha_{i}}|z_{\beta_{i-1}}|z_{\beta_{i}}}.
$$
\end{proposition}

\begin{proof}
We proceed by induction on $k+j$. If $k+j=0$, then, by Example~\ref{ejemplos de lambdas como productos}(1),
$$
\lambda_{a|b|c|d}^{e|e|g|g}= \lambda_{a | g\hs^{\bar{c}} | c | {}^{\bar{a}}\!e}^{e|e|g|g} \lambda_{g\hs^{\bar{c}} | b | {}^{\bar{a}}\!e |
d}^{e|e|g|g}.
$$
Assume $k+j>0$ and that the proposition holds for all pair of chains with the sum of the lengths smaller than $k+j$. Necessarily $z_{k+j-1}< h$ and $y_{k+j-1}=f$ or $z_{k+j-1}= h$ and $y_{k+j-1}<f$. In the first case $(a_i)_{i=0,\dots,k+j-1}$ is a configuration for $e=y_0<\dots < y_j=f$ and $g=z_0<z_1<\dots < z_{k-1}$ and so by inductive hypothesis
$$
\lambda_{a|z_{k-1}\hs^{\bar{c}}|c|{}^{\bar{a}}\!f }^{e|f|g|z_{k-1}}=\lambda_{a | g\hs^{\bar{c}} | c | {}^{\bar{a}}\! e}^{e|e|g|g} \lambda_{z_{k-1}\hs^{\bar{c}} | z_{k-1}\hs^{\bar{c}} | {}^{\bar{a}}\!f | {}^{\bar{a}}\!f }^{f|f|z_{k-1}|z_{k-1}} \prod_{i=1}^{k+j-1} \lambda_{z_{\beta_{i-1}}\hs^{\bar{c}} | z_{\beta_{i}}\hs^{\bar{c}} | {}^{\bar{a}}\!y_{\alpha_{i-1}} | {}^{\bar{a}}\!y_{\alpha_{i}}}^{y_{\alpha_{i-1}}|y_{\alpha_{i}}|z_{\beta_{i-1}}|z_{\beta_{i}}}.
$$
Since $\lambda_{z_{k-1}\hs^{\bar{c}} | z_{k-1}\hs^{\bar{c}} | {}^{\bar{a}}\!f | {}^{\bar{a}}\!f }^{f|f|z_{k-1}|z_{k-1}}=1$,
$$
\lambda_{a|b|c|d}^{e|f|g|h}= \lambda_{a | z_{k-1}\hs^{\bar{c}} | c | {}^{\bar{a}}\! f}^{e|f|g|z_{k-1}} \lambda_{z_{k-1}\hs^{\bar{c}} | b | {}^{\bar{a}}\!f | d}^{f|f|z_{k-1}|h}\quad\text{and}\quad\lambda_{z_{k-1}\hs^{\bar{c}} | b | {}^{\bar{a}}\!f | d}^{f|f|z_{k-1}|h}=\lambda_{z_{k-1}\hs^{\bar{c}} | h\hs^{\bar{c}} | {}^{\bar{a}}\!f | {}^{\bar{a}}\!f}^{f|f|z_{k-1}|h} \lambda_{h\hs^{\bar{c}} | b | {}^{\bar{a}}\!f | d}^{f|f|h|h},
$$
the result is true in this case. If $z_{k+j-1}= h$ and $y_{k+j-1}<f$, a similar argument proves the formula and concludes the proof.
\end{proof}

\begin{corollary}\label{unicidad}
Let $e\le f$, $g\le h$, $a\le b$ and $c\le d$, such that $a\hs^c\le g\le h\le b\hs^c$ and ${}^a\!c\le e \le f \le {}^a\!d$. The product
$$
\lambda_{a | g\hs^{\bar{c}} | c | {}^{\bar{a}}\! e}^{e|e|g|g}\lambda_{h\hs^{\bar{c}} | b | {}^{\bar{a}}\!f | d}^{f|f|h|h}\prod_{i=1}^{k+j}\lambda_i
$$
in Proposition~\ref{factorizacion} does not depend neither on the maximal chains nor on the chosen configuration.
\end{corollary}

\section{Construction of non-degenerate automorphisms}\label{seccion construccion de automorfismos no degenerados}
In Section~\ref{section: Properties} we proved that each non-degenerate coalgebra automorphism $r$ of $D\ot D$ induces by restriction a non-degenerate bijection
$$
r_{|} \colon X\times X \longrightarrow X\times X
$$
and fulfills  condition~\eqref{ordenes: r preserva la counidad} and the statements of Corollary~\ref{son iso de ordenes} and Propositions~\ref{donde puede no anularse} and~\ref{lambdas como productos}. In other words $r$ satisfies~\eqref{ordenes: r preserva la counidad} and, for all $a,b,c,d,e,f,g,h\in X$,

\begin{enumerate}[itemsep=1.0ex, topsep=1.0ex, label=\arabic*)]
\item the maps ${}^a\!(-)$ and $(-)\hs^b$, defined by $({}^a\!b,a\hs^b)\coloneqq r_{|}(a,b)$, are automorphisms of orders;

\item if $a$ and $b$ belong to the same component of $X$, then ${}^a\!(-)={}^b\!(-)$ and $(-)\hs^a=(-)\hs^b$;

\item if $a\le b$, $c\le d$, $e\le f$, $g\le h$ and $\lambda_{a|b|c|d}^{e|f|g|h}\ne 0$, then $a\hs^c\le g\le h\le b\hs^c$ and ${}^a\!c\le e \le f \le {}^a\!d$;

\item if $a\le b$, $c\le d$ $e\le f$, $g\le h$, $a\hs^c\le g\le h\le b\hs^c$ and ${}^a\!c\le e \le f \le {}^a\!d$, then
\begin{equation*}
\quad\qquad \lambda^{e|f|g|h}_{a|b|c|d} = \lambda_{a | z\hs^{\bar{c}} | c | {}^{\bar{a}}\!y}^{e|y|g|z} \lambda_{z\hs^{\bar{c}}| b | {}^{\bar{a}}\!y | d}^{y|f|z|h}
\end{equation*}
for each $y,z\in X$ such that $e\le y \le f$ and $g\le z \le h$;

\end{enumerate}

\begin{remark}\label{condicion para iso} Let $F_0\subseteq F_1\subseteq F_2\subseteq \dots$ be the filtration of $D\ot D$ defined setting $F_i$ as the $K$-subspace of $D\ot D$ generated by the tensors $(a,b)\ot (c,d)$ with $\mathfrak{h}[a,b]+\mathfrak{h}[c,d]\le i$. It is clear that a linear map $r\colon D\ot D\longrightarrow D\ot D$ that satisfies item~3) preserve this filtration. Assume that $r$ induces a non-degenerate permutation $r_{|}$ on $X\times X$. We claim that $r$ is bijective if and only if $\lambda_{a|b|c|d}^{{}^a\!c|{}^a\!d| a\hs^c|b\hs^c}\in K^{\times}$ for all~$(a,b),(c,d)\in Y$. In fact, in order to prove this it is sufficient to show that the last condition holds if and only if the graded morphism $r_l$ induced by $r$ is bijective. Now using again item~3) we obtain that
$$
r_l((a,b)\ot (c,d)) = \lambda_{a|b|c|d}^{{}^a\!c|{}^a\!d|a\hs^c|b\hs^c} ({}^a\!c,{}^a\!d)\ot (a\hs^c,b\hs^c).
$$
Consequently the condition is clearly necessary. The converse follows easily using that if $(e,g) = r_{|}(a,c) = ({}^a\!c,a\hs^c)$, then $((a,h\hs^{\bar{c}}),(c,{}^{\bar{a}}\! f))$ is the unique element of $Y\times Y$ such that $r_l((a,h\hs^{\bar{c}})\ot (c,{}^{\bar{a}}\! f))$ is a multiple of $(e,f)\ot (g,h)$ by a nonzero scalar.
\end{remark}

In this section we fix a linear automorphism  $r$ of $D\ot D$, and we prove that, conversely, if $r$ induces by restriction a non-degenerate bijection
$$
r_{|} \colon X\times X \longrightarrow X\times X
$$
and satisfies condition~\eqref{ordenes: r preserva la counidad} and items~1) -- 4), then $r$ is a non-degenerate coalgebra automorphism.

\begin{lemma}\label{suma nula}
Let $a,c,f,h\in X$. If $f\ge e\coloneqq {}^a\!c$, $h\ge g\coloneqq a^c$ and $\mathfrak{h}[e,f]+\mathfrak{h}[g,h]>0$, then
$$
\sum_{(p,q)\in [g\hs^{\bar{c}},h\hs^{\bar{c}}]\times [{}^{\bar{a}}\!e, {}^{\bar{a}}\! f]} \lambda_{g\hs^{\bar{c}}|p|{}^{\bar{a}}\!e|q}^{e|e|g|g}
\lambda_{p|h\hs^{\bar{c}}|q|{}^{\bar{a}}\! f}^{f|f|h|h} = 0.
$$
\end{lemma}

\begin{proof} For the sake of brevity we let $S$ denote the sum at the left hand of the above equality. We will proceed by induction on $N\coloneqq \mathfrak{h}[e,f]+\mathfrak{h}[g,h]$. So assume that the assertion holds when $\mathfrak{h}[e,f]+\mathfrak{h}[g,h]<N$. By Remark~\ref{ordenes: r sobre elementos de tipo grupo}
$$
S = \lambda_{g\hs^{\bar{c}}|h\hs^{ \bar{c}}|{}^{\bar{a}}\! e|{}^{\bar{a}}\! f}^{e|e|g|g} \, + \,\lambda_{g\hs^{\bar c}|h\hs^{\bar c}|{}^{\bar a}\! e|{}^{\bar a}\! f}^{f|f|h|h}\, + \sum_{\substack{(p,q)\in [g\hs^{\bar{c}},h\hs^{\bar{c}}]\times [{}^{\bar{a}}\!e, {}^{\bar{a}}\! f]\\ (p,q)\notin\{ (g\hs^{\bar c},{}^{\bar a}\! e),(h\hs^{\bar c},{}^{\bar a}\! f)\}}} \lambda_{g\hs^{\bar c}|p|{}^{\bar a}\! e|q}^{e|e|g|g} \lambda_{p|h\hs^{\bar c}|q|{}^{\bar a}\! f}^{f|f|h|h}.
$$
But, by condition~\eqref{ordenes: r preserva la counidad} and Statement~3), we have
$$
\lambda_{g\hs^{\bar{c}}|p|{}^{\bar{a}}\! e|q}^{e|e|g|g}=-\!\sum_{\substack{(i,j)\in [e,{}^a\!q] \times [g,p\hs^c]\\ (i,j)\ne(e,g)}}
\!\!\lambda_{g\hs^{\bar{c}}|p|{}^{\bar{a}}\!\! e|q}^{i|i|j|j}\quad\text{ and }\quad
\lambda_{p|h\hs^{\bar{c}}|q|{}^{\bar{a}}\! f}^{f|f|h|h}=-\!\!\sum_{\substack{(k,l)\in [{}^a\!q,f] \times [p^c,h]\\ (k,l)\ne(f,h)}}
\!\!\lambda_{p|h\hs^{\bar{c}}|q|{}^{\bar{a}}\!f}^{k|k|l|l},
$$
and so
\begin{equation}\label{0}
S=\lambda_{g\hs^{\bar c}|h\hs^{\bar c}|{}^{\bar a}\! e|{}^{\bar a}\! f}^{e|e|g|g}\, + \, \lambda_{g\hs^{\bar c}|h\hs^{\bar c}|{}^{\bar a}\! e|{}^{\bar a}\! f}^{f|f|h|h}\, + \sum_{\substack{i,j,p,q,k,l \in A \\ (i,j)\ne (e,g) , (k,l)\ne(f,h)}}  \lambda_{g\hs^{\bar{c}}|p|{}^{\bar a}\! e|q}^{i|i|j|j} \lambda_{p|h\hs^{\bar c}|q|{}^{\bar a}\! f}^{k|k|l|l},
\end{equation}
where $A\coloneqq \{i,j,p,q,k,l: g\le j \le p\hs^c \le l \le h \text{ and } e\le i \le {}^a\!q \le k \le f\}$. On the other hand
\begin{equation}\label{1}
\lambda_{g\hs^{\bar{c}}|h\hs^{ \bar{c}}|{}^{\bar{a}}\! e|{}^{\bar{a}}\! f}^{e|e|g|g} \, + \,\lambda_{g\hs^{\bar c}|h\hs^{\bar c}|{}^{\bar a}\! e|{}^{\bar a}\! f}^{f|f|h|h} \, + \, \sum_{\substack{(p,q)\in [g\hs^{\bar{c}},h\hs^{\bar{c}}]\times [{}^{\bar{a}}\!e, {}^{\bar{a}}\! f]\\ (p,q)\notin\{  (g\hs^{\bar c},{}^{\bar a}\! e),(h\hs^{\bar c},{}^{\bar a}\! f)\}}} \!\!\lambda_{g\hs^{\bar{c}}|p|{}^{\bar a}\!e|q}^{{}^a\!q|{}^a\!q|p\hs^c|p\hs^c} \lambda_{p|h\hs^{\bar c}|q|{}^{\bar a}\! f}^{{}^a\!q|{}^a\!q|p\hs^c|p\hs^c} = 0,
\end{equation}
since by condition~\eqref{ordenes: r preserva la counidad} and Statements~3) and 4),
$$
\lambda_{g\hs^{\bar{c}}|p|{}^{\bar{a}}\! e|q}^{{}^a\!q|{}^a\!q|p\hs^c|p\hs^c} \lambda_{p|h\hs^{\bar{c}}|q|{}^{\bar{a}}\!f}^{{}^a\!q|{}^a\!q|p\hs^c|p\hs^c} =\lambda_{g\hs^{\bar{c}}|h\hs^{\bar{c}}|{}^{\bar{a}}\!e|{}^{\bar{a}}\! f}^{{}^a\!q|{}^a\!q|p\hs^c|p\hs^c} \quad\text{and}\quad \sum_{\substack{\{p:g\hs^{\bar c}  \le p \le   h\hs^{\bar c}\}\\ \{q:{}^{\bar a}\! e\le  q \le {}^{\bar a}\! f\}}} \lambda_{g\hs^{\bar c}|h\hs^{\bar c}|{}^{\bar a}\!e|{}^{\bar a}\! f}^{^aq|^aq|p^c|p^c}=0.
$$
Combining equalities~\eqref{0} and~\eqref{1}, we obtain
\begin{align*}
S &=  \sum_{\substack{i,j,p,q,k,l \in A \\ (i,j)\ne (e,g) , (k,l)\ne(f,h) \\ \mathfrak{h}[i,k]+\mathfrak{h}[j,l]>0}}\!\! \lambda_{g\hs^{\bar{c}}|p|{}^{\bar{a}}\! e|q}^{i|i|j|j} \lambda_{p|h\hs^{\bar{c}}|q|{}^{\bar a}\! f}^{k|k|l|l}\\[3pt]
& = \sum_{\substack{\{j,l:g \le j  \le l\le h\}\\ \{i,k: e\le i\le k \le f\}\\ (i,j)\ne (e,g) , (k,l)\ne(f,h)\\ \mathfrak{h}[i,k]+\mathfrak{h}[j,l]>0}} \sum_{\substack{p\in [j\hs^{\bar c}, l\hs^{\bar c}]\\ q\in [{}^{\bar a}\!i, {}^{\bar a}\!k]}}
\lambda_{g\hs^{\bar{c}}|p|{}^{\bar{a}}\!e|q}^{i|i|j|j} \lambda_{p|h\hs^{\bar{c}}|q|{}^{\bar{a}}\!f}^{k|k|l|l}.
\end{align*}
But, since $(i,j)\ne ( e,g)$ and $(k,l)\ne( f, h)$ imply
$$
\mathfrak{h}[i,k]+\mathfrak{h}[j,l]< \mathfrak{h}[e,f]+\mathfrak{h}[g,h]=N,
$$
by the inductive hypothesis $S=0$, as desired.
\end{proof}

\begin{lemma}\label{suma nula mas general}
The map $r$ satisfies condition~\eqref{eq2}.
\end{lemma}

\begin{proof}
By conditions~2) and~3) we know that if $\lambda_{a|p|c|q}^{e|y|g|z} \lambda_{p|b|q|d}^{y'|f|z'|g} \ne 0$, then
$$
{}^a\!c \le e \le y \le {}^a\!q \le y' \le f \le  {}^a\!d \quad \text{and}\quad  a\hs^c \le g \le z \le p\hs^c \le z'\le g \le b\hs^c.
$$
So, we are reduced to prove that for each $a,b,c,d,e,f,g,h,y,y',z,z'\in X$ such that ${}^a\!c \le e \le y \le y' \le f \le  {}^a\!d$, $a\hs^c \le g \le z \le z'\le g \le b\hs^c$ and $(y,z) \ne (y',z')$.
$$
\sum_{(p,q)\in [z\hs^{\bar{c}},z'\hs^{\bar{c}}]\times [{}^{\bar{a}}\!y,{}^{\bar{a}}\!y']} \lambda_{a|p|c|q}^{e|y|g|z} \lambda_{p|b|q|d}^{y'|f|z'|h}=0.
$$
But by condition~4), we have
$$
\lambda_{a|p|c|q}^{e|y|g|z} = \lambda_{a | z\hs^{\bar{c}} | c | {}^{\bar{a}}\!y}^{e|y|g|z} \lambda_{z\hs^{\bar{c}}| p | {}^{\bar{a}}\!y | q}^{y|y|z|z}
\qquad\text{and}\qquad \lambda_{p|b|q|d}^{y'|f|z'|h} = \lambda_{p | z'\hs^{\bar{c}} | q | {}^{\bar{a}}\!y'}^{y'|y'|z'|z'} \lambda_{z'\hs^{\bar{c}}| b | {}^{\bar{a}}\!y' | d}^{y'|f|z'|h},
$$
and therefore it suffices to check that
$$
\sum_{(p,q)\in [z\hs^{\bar{c}},z'\hs^{\bar{c}}]\times [{}^{\bar{a}}\!y,{}^{\bar{a}}\!y']} \lambda_{z\hs^{\bar{c}}| p | {}^{\bar{a}}\!y | q}^{y|y|z|z}   \lambda_{p | z'\hs^{\bar{c}} | q | {}^{\bar{a}}\!y'}^{y'|y'|z'|z'}=0,
$$
which is true by Lemma~\ref{suma nula}.
\end{proof}

\begin{theorem}\label{principal para no degenerados}
Let $r \colon D\ot D \longrightarrow D\ot D$ be a linear map that induces by restriction a non-degenerate bijection $r_{|} \colon X\times X \longrightarrow X\times X$. Let
$$
\bigl(\lambda_{a|b|c|d}^{e|f|g|h}\bigr)_{(a,b),(c,d),(e,f),(g,h)\in Y}
$$
be as in the discussion above Remark~\ref{compatibilidad con epsilon}, and for each $a,c\in X$ let ${}^a \hf (-)$ and $(-)\hs^c$ be the maps introduced in Notation~\ref{coordenadas conjuntistas}. If $r$ satisfies items~1) -- 4) at the beginning of the section, condition~\eqref{ordenes: r preserva la counidad} and $\lambda_{a|b|c|d}^{{}^a\!c|{}^a\!d| a\hs^c|b\hs^c}\in K^{\times}$ for all~$(a,b),(c,d)\in Y$, then $r$ is a non-degenerate coalgebra automorphism.
\end{theorem}

\begin{proof} By Remark~\ref{condicion para iso} we know that the map $r$ is bijective. By hypothesis $r$ satisfies condition~\eqref{ordenes: r preserva la counidad}, and using items~2), 3) and 4), and arguing as in the proof of Proposition~\ref{lambdas como productos}, we obtain that $r$ satisfies condition~\eqref{eq1}. Moreover by Lemma~\ref{suma nula mas general} we know that $r$ also satisfies condition~\eqref{eq2}. Hence $r$ is a coalgebra automorphism and it only remains to check that it is non-degenerate. Let $G_l$ be the graded map induced by $(D \ot \sigma) \circ (\Delta \ot D)$. In order to prove that $(D \ot \sigma) \circ (\Delta \ot D)$ is invertible it suffices to show that so is $G_l$.  Let $a,b,c,d\in x$ with $a\le b$ and $c\le d$. A direct computation (using item~3)) shows that
$$
G_l((a,b) \ot (c,d)) = \lambda_{b|b|c|d}^{{}^a\!c | {}^a\!d | b\hs^c | b\hs^c} (a,b) \ot ({}^a\!c, {}^a\!d).
$$
So $G_l$ is invertible, because $\lambda_{b|b|c|d}^{{}^a\!c | {}^a\!d | b\hs^c | b\hs^c}\ne 0$ by hypothesis. A similar computation shows that the map $(\tau \ot D) \circ (D \ot \Delta)$ is also bijective and finishes the proof.
\end{proof}


\section{Non-degenerate solutions on incidence coalgebras}\label{section Non-degenerate solutions of the braid equation on coalgebras of orders}
In this section we assume that $r\colon D\ot D \longrightarrow D\ot D$ is a non-degenerate coalgebra automorphism that induces a non-degenerate  solution $r_{|}\colon X\times X \longrightarrow X \times X$ of the braid equation, and we determine necessary and sufficient conditions for $r$ to be a solution of the braid equation.

\begin{notations}
For all $a,b,c\in X$ in this section we set  $a\hs^b\hs^c \coloneqq (a\hs^b)\hs^c $, ${}^a\hf {}^b\hf c \coloneqq {}^a\hf ({}^b\hf c)$, ${}^{a\hs^b}\!c \coloneqq {}^{(a\hs^b)}\hf c$, $ a\hs^{{}^b\hf c}\coloneqq  a\hs^{({}^b\hf c)}$, ${}^a\hf b\hs^{{}^{a\hs^b}\hf c} \coloneqq ({}^a\hf b)\hs^{({}^{(a\hs^b)}\hf c)}$ and ${}^{c\hs^{{}^b\hf a}}\hf {b\hs^a} \coloneqq {}^{(c\hs^{({}^b\hf a)})}\hf (b\hs^a)$.
\end{notations}

\begin{remark}\label{sin mencion}
It is well known and easy to check that a permutation $(a,b) \mapsto ({}^a\hf b, a\hs^b)$ of $X\times X$
is a solution of the braid equation if and only if
$$
{}^a\hf {}^b\hf c = {}^{{}^a\hf b}\hspace{0.5pt}{}^{a\hs^b}\hf c,\quad a\hs^b\hs^c = a\hs^{{}^b\hf c} \hspace{0.5pt} {}^{b\hs^c}\quad\text{and}\quad  {}^{a\hs^{{}^b\hf c}}\hf {b\hs^c} = {}^a\hf b\hs^{{}^{a\hs^b}\hf c}\quad\text{for all $a,b,c\in X$.}
$$
\end{remark}

\begin{proposition}\label{condicion para braided}
The map $r$ is a solution of the braid equation if and only if for each family of six closed intervals $[a,b]$, $[c,d]$, $[e,f]$, $[g,h]$, $[i,j]$ and $[k,l]$, with
$$
[g,h] \subseteq [a,b],\quad [i,j] \subseteq [c,d]\quad\text{and}\quad [k,l] \subseteq [e,f],
$$
the following equality holds:
\begin{equation}
\begin{split}\label{eq braided}
\sum_{\substack{x\in [a,g] \\ y\in [h,b]}} \sum_{\substack{w\in [c,i] \\ z\in [j,d]}} & \sum_{\substack{u\in [e,k] \\ v\in [l,f]}} \lambda_{a|b|c|d}^{{}^a\!w | {}^a\!z | x\hs^c | y\hs^c} \lambda_{x\hs^c|y\hs^c|e|f}^{{}^{x\hs^c}\hf u | {}^{x\hs^c}\hf v  | g\hs^c\hs^e | h\hs^c\hs^e} \lambda_{{}^a\hf w|{}^a\hf z|{}^{x\hs^c}\hf u| {}^{x\hs^c}\hf v}^{{}^a\hf {}^c\hf k | {}^a\hf {}^c\hf l  |
{}^a\hf i\hs^{{}^{a\hs^i}\hf e}|{}^a\hf j\hs^{{}^{a\hs^j}\hf e}}\\
&= \sum_{\substack{x\in [a,g] \\ y\in [h,b]}} \sum_{\substack{w\in [c,i] \\ z\in [j,d]}} \sum_{\substack{u\in [e,k] \\ v\in [l,f]}} \lambda_{c|d|e|f}^{{}^c\!u | {}^c\!v | w\hs^e | z\hs^e}  \lambda_{a|b|{}^c\hf u|{}^c\hf v}^{{}^a\hf {}^c\hf k | {}^a\hf {}^c\hf l |  x\hs^{{}^c\hf u} | y\hs^{{}^c\hf u}} \lambda_{x\hs^{{}^c\hf u} | y\hs^{{}^c\hf u}|w\hs^e | z\hs^e}^{{}^{a\hs^{{}^i\hf e}}\hf {i\hs^e}| {}^{a\hs^{{}^j\hf e}}\hf {j\hs^e} | g\hs^c\hs^e | h\hs^c\hs^e}.
\end{split}
\end{equation}
\end{proposition}

\begin{proof}
A direct computation using that $r$ induces a solution $r_{|}$ of the set theoretic braid equation on $X\times X$, Proposition~\ref{donde puede no anularse} and Corollary~\ref{son iso de ordenes}, shows that
\begin{align*}
&(r \ot  D) \circ (D\ot r)\circ (r\ot D) ((a,b)\ot (c,d)\ot (e,f))\\[3pt]
& = \sum_{\substack{[x,y]\subseteq [a,b]\\ [w,z]\subseteq [c,d]}} \lambda_{a|b|c|d}^{{}^a\!w | {}^a\!z | x\hs^c | y\hs^c} (r\ot D) \circ (D\ot r) \bigl(({}^a\!w,{}^a\!z) \ot ( x\hs^c,y\hs^c) \ot (e,f)\bigr)\\[3pt]
& = \sum_{\substack{[x,y]\subseteq [a,b]\\ [w,z]\subseteq [c,d]}} \sum_{\substack{[g,h]\subseteq [x,y]\\ [u,v]\subseteq [e,f]}} \! \lambda_{a|b|c|d}^{{}^a\!w | {}^a\!z | x\hs^c | y\hs^c} \lambda_{x\hs^c|y\hs^c|e|f}^{{}^{x\hs^c}\!\hf u | {}^{x\hs^c}\!\hf v  | g\hs^c\hs^e | h\hs^c\hs^e} (r\ot D)\bigl(({}^a\!w,{}^a\!z) \ot ({}^{x\hs^c}\!\hf u , {}^{x\hs^c}\!\hf v) \ot (g\hs^c\hs^e , h\hs^c\hs^e)\bigr) \\[3pt]
& =  \sum_{\substack{[x,y]\subseteq [a,b]\\ [w,z]\subseteq [c,d]}} \sum_{\substack{[g,h]\subseteq [x,y]\\ [u,v]\subseteq [e,f]}} \sum_{\substack{[i,j]\subseteq [w,z]\\ [k,l]\subseteq [u,v]}} \lambda_{a|b|c|d}^{{}^a\!w | {}^a\!z | x\hs^c | y\hs^c} \lambda_{x\hs^c|y\hs^c|e|f}^{{}^{x\hs^c}\hf u | {}^{x\hs^c}\hf v  | g\hs^c\hs^e | h\hs^c\hs^e} \lambda_{{}^a\hf w|{}^a\hf z|{}^{x\hs^c}\hf u| {}^{x\hs^c}\hf v}^{{}^{{}^a\hf w}\hspace{0.5pt}{}^{x\hs^c}\hf k | {}^{{}^a\hf w}\hspace{0.5pt}{}^{x\hs^c}\hf l | {}^a\hf i\hs^{{}^{x\hs^c}\hf u} | {}^a\hf j \hs^{{}^{x\hs^c}\hf u}}\\[-5pt]
&\phantom{\sum_{\substack{[x,y]\subseteq [a,b]\\ [w,z]\subseteq [c,d]}} \sum_{\substack{[k,l]\subseteq [x,y]\\ [u,v]\subseteq [e,f]}} \sum_{\substack{[i,j]\subseteq [w,z]\\ [k,l]\subseteq [u,v]}}\lambda_{a|b|c|d}^{{}^a\!w | {}^a\!z | x\hs^c | y\hs^c}}\quad
\times ({}^{{}^a\hf w}\hspace{0.5pt}{}^{x\hs^c}\hf k , {}^{{}^a\hf w}\hspace{0.5pt}{}^{x\hs^c}\hf l)
\ot ({}^a\hf i\hs^{{}^{x\hs^c}\hf u} , {}^a\hf j \hs^{{}^{x\hs^c}\hf u}) \ot (g\hs^c\hs^e ,h\hs^c\hs^e)
\shortintertext{and}
& (D \ot r) \circ (r\ot D)\circ (D\ot r) ((a,b)\ot (c,d)\ot (e,f))\\[3pt]
&  = \sum_{\substack{[w,z]\subseteq [c,d]\\ [u,v]\subseteq [e,f]}} \lambda_{c|d|e|f}^{{}^c\!u | {}^c\!v | w\hs^e | z\hs^e} (D\ot r)\circ (r\ot D) \bigl((a,b)\ot ({}^c\!u , {}^c\!v)\ot (w\hs^e , z\hs^e)\bigr)\\[3pt]
&  = \sum_{\substack{[w,z]\subseteq [c,d]\\ [u,v]\subseteq [e,f]}} \sum_{\substack{[x,y]\subseteq [a,b]\\ [k,l]\subseteq [u,v]}} \! \lambda_{c|d|e|f}^{{}^c\!u | {}^c\!v | w\hs^e | z\hs^e}  \lambda_{a|b|{}^c\hf u|{}^c\hf v}^{{}^a\hf {}^c\hf k | {}^a\hf {}^c\hf l |  x\hs^{{}^c\hf u} | y\hs^{{}^c\hf u}}\! (D\ot r)\bigl(({}^a\hf {}^c\hf k , {}^a\hf {}^c\hf l) \ot (x\hs \hs^{{}^c\hf u} , y\hs \hs^{{}^c\hf u}) \ot (w\hs^e , z\hs^e) \hspace{-0.5pt}\bigr)\\[3pt]
& = \sum_{\substack{[w,z]\subseteq [c,d]\\ [u,v]\subseteq [e,f]}} \sum_{\substack{[x,y]\subseteq [a,b]\\ [k,l]\subseteq [u,v]}}  \sum_{\substack{[g,h]\subseteq [x,y]\\ [i,j]\subseteq [w,z]}} \lambda_{c|d|e|f}^{{}^c\!u | {}^c\!v | w\hs^e | z\hs^e}  \lambda_{a|b|{}^c\hf u|{}^c\hf v}^{{}^a\hf {}^c\hf k | {}^a\hf {}^c\hf l |  x\hs^{{}^c\hf u} | y\hs^{{}^c\hf u}} \lambda_{x\hs^{{}^c\hf u} | y\hs^{{}^c\hf u}|w\hs^e | z\hs^e}^{{}^{x\hs^{{}^c\hf u}}\hf {i\hs^e}| {}^{x\hs^{{}^c\hf u}}\hf {j\hs^e} | g\hs^{{}^c\hf u} \hspace{0.5pt} {}^{w\hs^e} | h\hs^{{}^c\hf u} \hspace{0.5pt} {}^{w\hs^e}} \\[-5pt]
&\phantom{\sum_{\substack{[x,y]\subseteq [a,b]\\ [w,z]\subseteq [c,d]}} \sum_{\substack{[g,h]\subseteq [x,y]\\ [u,v]\subseteq [e,f]}} \sum_{\substack{[i,j]\subseteq [w,z]\\ [k,l]\subseteq [u,v]}}\lambda_{a|b|c|d}^{{}^a\!w | {}^a\!z | x\hs^c | y\hs^c}}\quad\times ({}^a\hf {}^c\hf k , {}^a\hf {}^c\hf l) \ot ( {}^{x\hs^{{}^c\hf u}}\hf {i\hs^e}, {}^{x\hs^{{}^c\hf u}}\hf {j\hs^e}) \ot (g\hs^{{}^c\hf u} \hspace{0.5pt} {}^{w\hs^e} , h\hs^{{}^c\hf u} \hspace{0.5pt} {}^{w\hs^e}).
\end{align*}
Since, by Corollary~\ref{son iso de ordenes} and Remark~\ref{sin mencion}
\begin{align*}
& {}^{{}^a\hf w}\hspace{0.5pt}{}^{x\hs^c}\hf k  = {}^{{}^a\hf c}\hspace{0.5pt}{}^{a\hs^c}\hf k  = {}^a\hf {}^c\hf k,\quad {}^a\hf i\hs^{{}^{x\hs^c}\hf u} = {}^a\hf i\hs^{{}^{a\hs^i}\hf e}  =  {}^{a\hs^{{}^i\hf e}}\hf {i\hs^e} = {}^{x\hs^{{}^c\hf u}}\hf {i\hs^e} ,\quad  g\hs^c\hs^e  = g\hs^{{}^c\hf e} \hspace{0.5pt} {}^{c\hs^e}  = g\hs^{{}^c\hf u} \hspace{0.5pt} {}^{w\hs^e}, \\
& {}^{{}^a\hf w}\hspace{0.5pt}{}^{x\hs^c}\hf l  = {}^{{}^a\hf c}\hspace{0.5pt}{}^{a\hs^c}\hf l  = {}^a\hf {}^c\hf l,\quad {}^a\hf j\hs^{{}^{x\hs^c}\hf u} = {}^a\hf j\hs^{{}^{a\hs^j}\hf e}  =  {}^{a\hs^{{}^j\hf e}}\hf {j\hs^e} = {}^{x\hs^{{}^c\hf u}}\hf {j\hs^e} ,\quad  h\hs^c\hs^e  = h\hs^{{}^c\hf e} \hspace{0.5pt} {}^{c\hs^e}  = h\hs^{{}^c\hf u} \hspace{0.5pt} {}^{w\hs^e},
\end{align*}
the result follows immediately from the above equalities.
\end{proof}

\subsection{Small intervals}\label{small intervals}

Next we analyze exhaustively the meaning of equations~\eqref{eq braided} when the sum of the lengths of the intervals $[a,b]$, $[c,d]$ and $[e,f]$ is
smaller than or equal to 1:

\begin{enumerate}[itemsep=1.0ex, topsep=1.0ex, label=\arabic*)]

\item When $a=b$, $c=d$ and $e=f$ this equation reduces to
\begin{multline*}
\quad\qquad\lambda_{a|a|c|c}^{{}^a\!c | {}^a\!c | a\hs^c | a\hs^c} \lambda_{a\hs^c|a\hs^c|e|e}^{{}^{a\hs^c}\hf e | {}^{a\hs^c}\hf e  | a\hs^c\hs^e | a\hs^c\hs^e} \lambda_{{}^a\hf c|{}^a\hf c|{}^{a\hs^c}\hf e| {}^{a\hs^c}\hf e}^{{}^a\hf {}^c\hf e | {}^a\hf {}^c\hf e  | {}^{a\hs^{{}^c\hf e}}\hf {c\hs^e} | {}^{a\hs^{{}^c\hf e}}\hf {c\hs^e}} \\
=\lambda_{c|c|e|e}^{{}^c\!e | {}^c\!e | c\hs^e | c\hs^e}  \lambda_{a|a|{}^c\hf e|{}^c\hf e}^{{}^a\hf {}^c\hf e | {}^a\hf {}^c\hf e |  a\hs^{{}^c\hf e} | a\hs^{{}^c\hf e}} \lambda_{a\hs^{{}^c\hf e} | a\hs^{{}^c\hf e}|c\hs^e | c\hs^e}^{{}^{a\hs^{{}^c\hf e}}\hf {c\hs^e}| {}^{a\hs^{{}^c\hf e}}\hf {c\hs^e} | a\hs^c\hs^e | a\hs^c\hs^e}.
\end{multline*}
This is true since the expressions at the both sides of the equal sign are~$1$.

\item When $a=g=h \prec b$, $c=d$ and $e=f$, it reduces to
\begin{multline*}
\quad\qquad \lambda_{a|b|c|c}^{{}^a\!c | {}^a\!c | a\hs^c | a\hs^c} + \lambda_{a|b|c|c}^{{}^a\!c | {}^a\!c | a\hs^c | b\hs^c} \lambda_{a\hs^c|b\hs^c|e|e}^{{}^{a\hs^c}\hf e | {}^{a\hs^c}\hf e  | a\hs^c\hs^e | a\hs^c\hs^e} \\
= \lambda_{a|b|{}^c\hf e|{}^c\hf e}^{{}^a\hf {}^c\hf e | {}^a\hf {}^c\hf e |  a\hs^{{}^c\hf e} | a\hs^{{}^c\hf e}} + \lambda_{a|b|{}^c\hf e|{}^c\hf e}^{{}^a\hf {}^c\hf e | {}^a\hf {}^c\hf e |  a\hs^{{}^c\hf e} | b\hs^{{}^c\hf e}} \lambda_{a\hs^{{}^c\hf e} | b\hs^{{}^c\hf e}|c\hs^e | c\hs^e}^{{}^{a\hs^{{}^c\hf e}}\hf {c\hs^e}| {}^{a\hs^{{}^c\hf e}}\hf {c\hs^e} | a\hs^c\hs^e | a\hs^c\hs^e}.
\end{multline*}

\item When $a=g \prec h = b$, $c=d$ and $e=f$, it reduces to
$$
\lambda_{a|b|c|c}^{{}^a\!c | {}^a\!c | a\hs^c | b\hs^c}  \lambda_{a\hs^c|b\hs^c|e|e}^{{}^{a\hs^c}\hf e | {}^{a\hs^c}\hf e  | a\hs^c\hs^e | b\hs^c\hs^e} = \lambda_{a|b|{}^c\hf e|{}^c\hf e}^{{}^a\hf {}^c\hf e | {}^a\hf {}^c\hf e |  a\hs^{{}^c\hf e} | b\hs^{{}^c\hf e}} \lambda_{a\hs^{{}^c\hf e} | b\hs^{{}^c\hf e}|c\hs^e | c\hs^e}^{{}^{a\hs^{{}^c\hf e}}\hf {c\hs^e}| {}^{a\hs^{{}^c\hf e}}\hf {c\hs^e} | a\hs^c\hs^e | b\hs^c\hs^e}.
$$

\item When $a \prec g = h = b$, $c=d$ and $e=f$, it reduces to
\begin{multline*}
\quad\qquad \lambda_{a|b|c|c}^{{}^a\!c | {}^a\!c | b\hs^c | b\hs^c} + \lambda_{a|b|c|c}^{{}^a\!c | {}^a\!c | a\hs^c | b\hs^c} \lambda_{a\hs^c|b\hs^c|e|e}^{{}^{a\hs^c}\hf e | {}^{a\hs^c}\hf e  | b\hs^c\hs^e | b\hs^c\hs^e} \\
=  \lambda_{a|b|{}^c\hf e|{}^c\hf e}^{{}^a\hf {}^c\hf e | {}^a\hf {}^c\hf e |  b\hs^{{}^c\hf e} | b\hs^{{}^c\hf e}} + \lambda_{a|b|{}^c\hf e|{}^c\hf e}^{{}^a\hf {}^c\hf e | {}^a\hf {}^c\hf e |  a\hs^{{}^c\hf e} | b\hs^{{}^c\hf e}} \lambda_{a\hs^{{}^c\hf e} | b\hs^{{}^c\hf e}|c\hs^e | c\hs^e}^{{}^{a\hs^{{}^c\hf e}}\hf {c\hs^e}| {}^{a\hs^{{}^c\hf e}}\hf {c\hs^e} | b\hs^c\hs^e | b\hs^c\hs^e},
\end{multline*}
which is equivalent to the condition obtained in item~2), because $r$ satisfies condition~\eqref{ordenes: r preserva la counidad} and the condition required in item~3) at the beginning of Section~\ref{seccion construccion de automorfismos no degenerados} is fulfilled.

\item When $a= b$, $c=i=j \prec d$ and $e=f$, it reduces to
\begin{multline*}
\quad\qquad \lambda_{a|a|c|d}^{{}^a\!c | {}^a\!c | a\hs^c | a\hs^c}  + \lambda_{a|a|c|d}^{{}^a\!c | {}^a\!d | a\hs^c | a\hs^c} \lambda_{{}^a\hf c|{}^a\hf d|{}^{a\hs^c}\hf e| {}^{a\hs^c}\hf e}^{{}^a\hf {}^c\hf e | {}^a\hf {}^c\hf e  |  {}^a\hf c\hs^{{}^{a\hs^c}\hf e}| {}^a\hf c\hs^{{}^{a\hs^c}\hf e}}\\
=  \lambda_{c|d|e|e}^{{}^c\!e | {}^c\!e | c\hs^e | c\hs^e}  + \lambda_{c|d|e|e}^{{}^c\!e | {}^c\!e | c\hs^e | d\hs^e} \lambda_{a\hs^{{}^c\hf e} | a\hs^{{}^c\hf e}|c\hs^e | d\hs^e}^{{}^{a\hs^{{}^c\hf e}}\hf {c\hs^e}| {}^{a\hs^{{}^c\hf e}}\hf {c\hs^e} | a\hs^c\hs^e | a\hs^c\hs^e}.
\end{multline*}

\item When $a= b$, $c=i \prec j = d$ and $e=f$, it reduces to
$$
\quad\qquad \lambda_{a|a|c|d}^{{}^a\!c | {}^a\!d | a\hs^c | a\hs^c} \lambda_{{}^a\hf c|{}^a\hf d|{}^{a\hs^c}\hf e| {}^{a\hs^c}\hf e}^{{}^a\hf {}^c\hf e | {}^a\hf {}^c\hf e  | {}^a\hf c\hs^{{}^{a\hs^c}\hf e}| {}^a\hf d\hs^{{}^{a\hs^d}\hf e}}
= \lambda_{c|d|e|e}^{{}^c\!e | {}^c\!e | c\hs^e | d\hs^e}  \lambda_{a\hs^{{}^c\hf e} | a\hs^{{}^c\hf e}|c\hs^e | d\hs^e}^{{}^{a\hs^{{}^c\hf e}}\hf {c\hs^e}| {}^{a\hs^{{}^d\hf e}}\hf {d\hs^e} | a\hs^c\hs^e | a\hs^c\hs^e}.
$$

\item When $a= b$, $c\prec i = j = d$ and $e=f$, it reduces to
\begin{multline*}
\quad \qquad \lambda_{a|a|c|d}^{{}^a\!d | {}^a\!d | a\hs^c | a\hs^c} + \lambda_{a|a|c|d}^{{}^a\!c | {}^a\!d | a\hs^c | a\hs^c} \lambda_{{}^a\hf c|{}^a\hf d|{}^{a\hs^c}\hf e| {}^{a\hs^c}\hf e}^{{}^a\hf {}^c\hf e | {}^a\hf {}^c\hf e  | {}^a\hf d\hs^{{}^{a\hs^d}\hf e}| {}^a\hf d\hs^{{}^{a\hs^d}\hf e}}\\
%
=  \lambda_{c|d|e|e}^{{}^c\!e | {}^c\!e | d\hs^e | d\hs^e} + \lambda_{c|d|e|e}^{{}^c\!e | {}^c\!e | c\hs^e | d\hs^e}   \lambda_{a\hs^{{}^c\hf e} | a\hs^{{}^c\hf e}|c\hs^e | d\hs^e}^{{}^{a\hs^{{}^d\hf e}}\!\hf {d\hs^e}| {}^{a\hs^{{}^d\hf e}}\!\hf {d\hs^e} | a\hs^c\hs^e | a\hs^c\hs^e} ,
\end{multline*}
which is equivalent to the condition obtained in item~5), by the same argument as in item~4).

\item When $a= b$, $c= d$ and $e=k=l \prec f$, it reduces to
\begin{multline*}
\quad \qquad \lambda_{a\hs^c|a\hs^c|e|f}^{{}^{a\hs^c}\hf e | {}^{a\hs^c}\hf e  | a\hs^c\hs^e | a\hs^c\hs^e} +  \lambda_{a\hs^c|a\hs^c|e|f}^{{}^{a\hs^c}\hf e | {}^{a\hs^c}\hf f  | a\hs^c\hs^e | a\hs^c\hs^e} \lambda_{{}^a\hf c|{}^a\hf c|{}^{a\hs^c}\hf e| {}^{a\hs^c}\hf f}^{{}^a\hf {}^c\hf e | {}^a\hf {}^c\hf e  | {}^{a\hs^{{}^c\hf e}}\!\hf {c\hs^e} | {}^{a\hs^{{}^c\hf e}}\!\hf {c\hs^e}}\\
= \lambda_{c|c|e|f}^{{}^c\!e | {}^c\!e | c\hs^e | c\hs^e}  + \lambda_{c|c|e|f}^{{}^c\!e | {}^c\!f | c\hs^e | c\hs^e}  \lambda_{a|a|{}^c\hf e|{}^c\hf f}^{{}^a\hf {}^c\hf e | {}^a\hf {}^c\hf e |  a\hs^{{}^c\hf e} | a\hs^{{}^c\hf e}}.
\end{multline*}

\item When $a= b$, $c= d$ and $e=k\prec l=f$, it reduces to
$$
\quad\qquad  \lambda_{a\hs^c|a\hs^c|e|f}^{{}^{a\hs^c}\hf e | {}^{a\hs^c}\hf f  | a\hs^c\hs^e | a\hs^c\hs^e} \lambda_{{}^a\hf c|{}^a\hf c|{}^{a\hs^c}\hf e| {}^{a\hs^c}\hf f}^{{}^a\hf {}^c\hf e | {}^a\hf {}^c\hf f  | {}^{a\hs^{{}^c\hf e}}\!\hf {c\hs^e} | {}^{a\hs^{{}^c\hf e}}\!\hf {c\hs^e}}= \lambda_{c|c|e|f}^{{}^c\!e | {}^c\!f | c\hs^e |c\hs^e}  \lambda_{a|a|{}^c\hf e|{}^c\hf f}^{{}^a\hf {}^c\hf e | {}^a\hf {}^c\hf f |  a\hs^{{}^c\hf e} | a\hs^{{}^c\hf e}}.
$$

\item When $a= b$, $c= d$ and $e \prec k = l = f$, it reduces to
\begin{multline*}
\quad\qquad  \lambda_{a\hs^c|a\hs^c|e|f}^{{}^{a\hs^c}\hf f | {}^{a\hs^c}\hf f  | a\hs^c\hs^e | a\hs^c\hs^e} + \lambda_{a\hs^c|a\hs^c|e|f}^{{}^{a\hs^c}\hf e | {}^{a\hs^c}\hf f  | a\hs^c\hs^e | a\hs^c\hs^e} \lambda_{{}^a\hf c|{}^a\hf c|{}^{a\hs^c}\hf e| {}^{a\hs^c}\hf f}^{{}^a\hf {}^c\hf f | {}^a\hf {}^c\hf f  | {}^{a\hs^{{}^c\hf e}}\hf {c\hs^e} | {}^{a\hs^{{}^c\hf e}}\hf {c\hs^e}}\\
= \lambda_{c|c|e|f}^{{}^c\hf f | {}^c\hf f | c\hs^e |c\hs^e} + \lambda_{c|c|e|f}^{{}^c\!e | {}^c\!f | c\hs^e |c\hs^e}  \lambda_{a|a|{}^c\hf e|{}^c\hf f}^{{}^a\hf {}^c\hf f | {}^a\hf {}^c\hf f |  a\hs^{{}^c\hf e} | a\hs^{{}^c\hf e}},
\end{multline*}
which is equivalent to the condition obtained in item~8), by the same argument as in item~4).
\end{enumerate}

In the sequel for $v=(v_1,v_2)$ and $w=(w_1,w_2)$ in $K^2$ we say that $v$ and $w$ are aligned and we write $v\sim w$ if $\det\begin{psmallmatrix} v_1&v_2\\ w_1&w_2 \end{psmallmatrix}=0$.

\begin{example}
Let $X_{a,h}\subseteq X$ be the set $\{a,b,c,d,e,f,g,h\}$. Assume that $r_{|}$ is the flip on $X_{a,h}\times X_{a,h}$. Then items~3), 6) and 9) are automatically fulfilled; whereas items~2), 5) and~8) say that when $a \prec b$,  $c=d$ and $e=f$,
\begin{equation}\label{alineados 1}
\lambda_{a|b|c|c}^{c|c|a|a} + \lambda_{a|b|c|c}^{c|c|a|b}\lambda_{a|b|e|e}^{e|e|a|a} = \lambda_{a|b|e|e}^{e|e|a|a} + \lambda_{a|b|e|e}^{e|e|a|b}\lambda_{a|b|c|c}^{c|c|a|a},
\end{equation}
that when $a=b$, $c \prec d$ and $e=f$,
\begin{equation}\label{alineados 2}
\lambda_{a|a|c|d}^{c|c|a|a} + \lambda_{a|a|c|d}^{c|d|a|a}\lambda_{c|d|e|e}^{e|e|c|c} = \lambda_{c|d|e|e}^{e|e|c|c} + \lambda_{c|d|e|e}^{e|e|c|d}\lambda_{a|a|c|d}^{c|c|a|a},
\end{equation}
and that when $a=b$, $c=d$ and $e\prec f$,
\begin{equation}\label{alineados 3}
\lambda_{a|a|e|f}^{e|e|a|a} + \lambda_{a|a|e|f}^{e|f|a|a}\lambda_{c|c|e|f}^{e|e|c|c} = \lambda_{c|c|e|f}^{e|e|c|c} + \lambda_{c|c|e|f}^{e|f|c|c}\lambda_{a|a|e|f}^{e|e|a|a}.
\end{equation}
Equality~\eqref{alineados 1} says that
$$
(\lambda_{a|b|e|e}^{e|e|a|a}, \lambda_{a|b|e|e}^{e|e|a|b} - 1) \sim (\lambda_{a|b|c|c}^{c|c|a|a}, \lambda_{a|b|c|c}^{c|c|a|b} - 1),
$$
equality~\eqref{alineados 2} says that
$$
(\lambda_{c|d|e|e}^{e|e|c|c}, \lambda_{c|d|e|e}^{e|e|c|d} - 1)\sim (\lambda_{a|a|c|d}^{c|c|a|a}, \lambda_{a|a|c|d}^{c|d|a|a} - 1)
$$
and equality~\eqref{alineados 3} says that
$$
(\lambda_{c|c|e|f}^{e|e|c|c}, \lambda_{c|c|e|f}^{e|f|c|c} - 1)\sim(\lambda_{a|a|e|f}^{e|e|a|a}, \lambda_{a|a|e|f}^{e|f|a|a} - 1).
$$
\end{example}

Let $r\colon D\ot D \longrightarrow D\ot D$ be a non-degenerate coalgebra automorphism that induces a non-degenerate  solution $r_{|}\colon X\times X \longrightarrow X \times X$ of the set theoretic braid equation. Assume that there exist two commuting order automorphisms $\phi_r,\phi_l\colon X\to X$ such that
$\phi_r={}^x\hf(-)$ and $\phi_l= (-)\hs^y$ for all $x,y\in X$. For all $s,a,b\in X$ with $a\prec b$ and $i\in \mathds{Z}$, we will write
\begin{alignat}{2}
&s^{(i)}\coloneqq \phi_r^{i}(s),&&\qquad {}^{(i)}\hf s\coloneqq \phi_l^i(s), \label{potencia de phi}\\
&\alpha_r(s)(a,b)\coloneqq\lambda_{a|b|s|s}^{{}^{(1)}\hf s | {}^{(1)}\hf s | a^{(1)} | b^{(1)}},&&\qquad
\beta_r(s)(a,b)\coloneqq\lambda_{a|b|s|s}^{ {}^{(1)}\hf s | {}^{(1)}\hf s | a^{(1)} | a^{(1)}},\label{definicion de a sub r}\\
&\alpha_l(s)(a,b)\coloneqq\lambda_{s|s|a|b}^{ {}^{(1)}\hf a | {}^{(1)}\hf b | s^{(1)} | s^{(1)}},&&\qquad
\beta_l(s)(a,b)\coloneqq\lambda_{s|s|a|b}^{{}^{(1)}\hf a | {}^{(1)}\hf a | s^{(1)} |s^{(1)}}.\label{definicion de a sub l}
\end{alignat}
For the sake of brevity in the following result we write
\begin{equation}\label{a sub l a la i y a sub r a la i}
\alpha_r^{(i)}(s)\coloneqq \alpha_r(s^{(i)})(a^{(i)},b^{(i)}) \quad\text{and}\quad \alpha_l^{(i)}(s)\coloneqq \alpha_l({}^{(i)}s)({}^{(i)}a,{}^{(i)}b),
\end{equation}
and we define $\beta_r^{(i)}(s)$ and $\beta^{(i)}_l(s)$ in a similar way. The following proposition generalizes the result obtained in the previous example.

\begin{proposition}\label{parte lineal}
Let $n\in \mathds{N}$. Assume that $\phi_r^{n}=\phi_l^{n}=\ide$ and that each element of~$K^{\times}$ has $n$ distinct $n$th roots, and fix a primitive $n$th root of unity $w$. The following facts hold:

\begin{enumerate}[itemsep=1.0ex, topsep=1.0ex, label=\emph{\arabic*)}]

\item Item~3) of Subsection~\ref{small intervals} is satisfied if and only if for all $a\prec b$ in $X$ there exists a constant $C_r(a,b)\in K^{\times}$ such that
\begin{equation*}\label{cocientes alpha mas condicion}
\qquad\quad\frac{\alpha_r^{(1)}\!(s)}{\alpha_r(s)}=C_r(a,b) \quad \text{and}\quad \alpha_r(s)=\alpha_r({}^{(1)}\hf s \hs^{(1)})\,\text{ for all $s\in X$.}
\end{equation*}

\item Item~9) of Subsection~\ref{small intervals} is satisfied if and only if for all $a\prec b$ in $X$ there exists a constant $C_l(a,b)\in K^{\times}$ such that
\begin{equation*}\label{cocientes alpha mas condicion nueve}
\qquad\quad \frac{\alpha_l^{(1)}\!(s)}{\alpha_l(s)}=C_l(a,b) \quad \text{and}\quad \alpha_l(s)=\alpha_l({}^{(1)}\hf s \hs^{(1)})\,\text{ for all $s\in X$.}
\end{equation*}

\item Assume that the conditions in item~1) are fulfilled. Then item~2) of Subsection~\ref{small intervals} is satisfied if and only if for all $s\in X$
and $a\prec b$ in $X$
\begin{align*}
&\beta_r(s)=\beta_r({}^{(1)}\hf s \hs^{(1)}),
\shortintertext{and for all $a,b,s,t\in X$ with $a\prec b$ and each $0\le i<n$}
\qquad\quad& \biggl(\gamma_r\alpha_r(s)-w^i,\sum_{j=0}^{n-1}\wp_j w^{ij}\beta_r^{(j)}(s)\biggr)\sim \biggl(\gamma_r\alpha_r(t)-w^i,\sum_{j=0}^{n-1}\wp_j w^{ij}\beta_r^{(j)}(t)\biggr),
\end{align*}
where

\begin{itemize}[itemsep=1.0ex, topsep=1.0ex]

\item[-] $\gamma_r$ is a fixed $n$th root of $\displaystyle\prod_{u=0}^{n-2} C_r(a^{(u)},b^{(u)})^{n-u-1}$,

\item[-] $\wp_j\coloneqq \displaystyle\frac{1}{\gamma_r^{j+1}} \prod_{u=0}^{n-2}C_r(a^{(u)},b^{(u)}) \prod_{u=0}^{j-2} C_r(a^{(u)},b^{(u)})^{j-u-1}$.
\end{itemize}

\item Assume that the conditions in item~2) are fulfilled. Then item~8) of Subsection~\ref{small intervals} is satisfied if and only if for all $s\in X$
and $a\prec b$ in $X$
\begin{align*}
&\beta_l(s)=\beta_l({}^{(1)}\hf s \hs^{(1)}),
\shortintertext{and for all $a,b,s,t\in X$ with $a\prec b$ and each $0\le i<n$}
\qquad\quad& \biggl(\gamma_l\alpha_l(s)-w^i,\sum_{j=0}^{n-1}\ell_j w^{ij}\beta_l^{(j)}(s)\biggr)\sim \biggl(\gamma_l\alpha_l(t)-w^i,\sum_{j=0}^{n-1} \ell_j w^{ij}\beta_l^{(j)}(t)\biggr),
\end{align*}
where

\begin{itemize}[itemsep=1.0ex, topsep=1.0ex]

\item[-] $\gamma_l$ is a fixed $n$th root of $\displaystyle\prod_{u=0}^{n-2} C_l({}\hs^{(u)}a,{}\hs^{(u)}b)^{n-u-1}$,

\item[-] $\ell_j\coloneqq \displaystyle\frac{1}{\gamma_l^{j+1}} \prod_{u=0}^{n-2}C_l({}\hs^{(u)}a,{}\hs^{(u)}b) \prod_{u=0}^{j-2} C_l({}\hs^{(u)}a,{} \hs^{(u)}b)^{j-u-1}$.
\end{itemize}

\end{enumerate}
\end{proposition}

\begin{proof} Assume that the equality in item~3) of Subsection~\ref{small intervals} holds. By Remark~\ref{condicion para iso} all terms in that equality are non zero. Replacing $e$ by $s^{(1)}$ in it, we obtain
$$
\frac{\alpha_r^{(1)}(c)}{\alpha_r(c)} = \frac{\alpha_r^{(1)}(s)}{\alpha_r({}^{(1)}s^{(1)})} = \frac{\alpha_r^{(1)}(s)}{\alpha_r(s)}\qquad\text{for all $c,s\in X$,}
$$
where the last equality follows from the first one taking $c=s$. From this it follows immediately that there exists $C_r(a,b)\in K^{\times}$ such that the equalities in item~1) are true. The converse is straightforward. A similarly argument proves item~2).

\smallskip

Assume now that the conditions in item~1) are fulfilled and that the equality in item~2) of Subsection~\ref{small intervals} holds. By Remark~\ref{condicion para iso}, setting $(a,b)\coloneqq~(a\hs^{(-1)},b\hs^{(-1)})$ and $c\coloneqq{}^{(1)}\hf e$ the equality yields $\beta_r(e)=\beta_r({}^{(1)}\hf e \hs^{(1)})$ for all $e\in X$. Using the same equality with $c\coloneqq s^{(i)}$, $(a,b)\coloneqq(a\hs^{(i)},b\hs^{(i)})$ and $e\coloneqq t^{(i+1)}$, where $i\in \{0,\dots,n-1\}$, we obtain
\begin{equation}\label{igualdades}
\begin{aligned}
\beta_r^{(0)}(s)+\alpha_r^{(0)}(s)\beta_r^{(1)}(t) & =\beta_r^{(0)}(t)+\alpha_r^{(0)}(t)\beta_r^{(1)}(s)  \\
\beta_r^{(1)}(s)+\alpha_r^{(1)}(s)\beta_r^{(2)}(t) &=\beta_r^{(1)}(t)+\alpha_r^{(1)}(t)\beta_r^{(2)}(s)  \\
&\vdotswithin{=}\\
\beta_r^{(n-1)}(s)+\alpha_r^{(n-1)}(s)\beta_r^{(0)}(t) &=\beta_r^{(n-1)}(t)+\alpha_r^{(n-1)}(t)\beta_r^{(0)}(s),
\end{aligned}
\end{equation}
where in the last equation we have used that $\beta_r^{(n)}=\beta_r^{(0)}$. For each $j\in \mathds{N}_0$, let
$$
C(j)\coloneqq \prod_{u=0}^{j-1} C_r(a^{(u)},b^{(u)}).
$$
Using that $\wp_j C(j)=\gamma_r \wp_{j+1}$ and $\wp_n=\wp_0$ we obtain that
\begin{equation}\label{sumas}
\sum_{j=0}^{n-1}w^{ij} \wp_j C(j) \beta_r^{(j+1)}(x)=\sum_{j=0}^{n-1}w^{ij} \gamma_r\wp_{j+1} \beta_r^{(j+1)}(x) =\frac{\gamma_r}{w^i} \sum_{j=0}^{n-1}w^{ij} \wp_j \beta_r^{(j)}(x),
\end{equation}
for all $x\in X$ and $i\in\{0,\dots,n-1\}$.
Adding the first equality in~\eqref{igualdades} multiplied by $\wp_0$ to the second one multiplied by $w^i\wp_1$, and so on until we add the last equality multiplied by $w^{i(n-1)}\wp_{n-1}$, and using that $\alpha_r^{(j)}(s)=\alpha_r(s)C(j)$, we obtain
\begin{equation}\label{igualdad de sumas}
\begin{split}
\sum_{j=0}^{n-1}w^{ij} \wp_j\beta_r^{(j)}(s)&+\alpha_r(s)\sum_{j=0}^{n-1}w^{ij} \wp_j C(j) \beta_r^{(j+1)}(t)\\
&=  \sum_{j=0}^{n-1}w^{ij} \wp_j\beta_r^{(j)}(t)+\alpha_r(t)\sum_{j=0}^{n-1}w^{ij} \wp_j C(j) \beta_r^{(j+1)}(s),
\end{split}
\end{equation}
for $i\in\{0,\dots,n-1\}$. Hence, by~\eqref{sumas},
\begin{equation}\label{igualdad con S}
S_r^{(i)}(s)+\alpha_r(s)\frac{\gamma_r}{w^i} S_r^{(i)}(t)=S_r^{(i)}(t)+\alpha_r(t)\frac{\gamma_r}{w^i} S_r^{(i)}(s),
\end{equation}
where $S_r^{(i)}(x)\coloneqq \sum_{j=0}^{n-1}w^{ij} \wp_j\beta_r^{(j)}(x)$ for $x\in X$, and so, for $0\le i< n$, we have
\begin{equation}\label{equivalencia a derecha}
(\gamma_r \alpha_r(s)-w^i,S_r^{(i)}(s))\sim (\gamma_r \alpha_r(t)-w^i,S_r^{(i)}(t))\quad\text{for all $s,t\in X$,}
\end{equation}
as desired.

Conversely assume that $\beta_r(e)=\beta_r({}^{(1)}\hf e \hs^{(1)})$ for all $e\in X$, and that~\eqref{equivalencia a derecha}
holds, which means that~\eqref{igualdad con S} holds. By~\eqref{sumas} the systems~\eqref{igualdad con S} and~\eqref{igualdad de sumas} are equivalent. We claim that the systems~\eqref{igualdad de sumas} and~\eqref{igualdades} are also equivalent. Indeed, this follows easily from the fact that all the $\wp_j$'s are non zero and that the matrix
$$
\begin{pmatrix}1&1&\dots&1\\
1&w&\dots&w^{n-1}\\
\vdots&\vdots&\ddots&\vdots\\
1&w^{n-1}&\dots&w^{(n-1)^2}
\end{pmatrix}
$$
is invertible, because it is the Vandermonde matrix associated with the elements $1,w,w^2,\dots,w^{n-1}$, which are all dif\-fe\-rent.
Item~2) of Subsection~\ref{small intervals} follows immediately from the first equality in~\eqref{igualdades} with $s$ replaced by $c$ and $t$ replaced by
${}^{(1)} e$, using
that $\beta_r({}^{(1)}\hf e \hs^{(1)})(a^{(1)},b^{(1)}) = \beta_r(e)(a^{(1)},b^{(1)})$. A similar argument proves item~(4).
\end{proof}

Let $\phi_r,\phi_l,\alpha_r^{(i)},\alpha_l^{(i)},\beta_r^{(i)}$ and $\beta_l^{(i)}$ be as in the discussion above Proposition~\ref{parte lineal}.

\begin{proposition}\label{corolario parte lineal}
Let $n\in \mathds{N}$. Assume that $\phi_r=\phi_l$, that $\phi_r^{n}=\ide$ and that each element of~$K^{\times}$ has $n$ distinct $n$-roots and fix a primitive $n$-root of unity $w$. Then equality~\eqref{eq braided} is satisfied for all the intervals $[a,b]$, $[c,d]$, $[e,f]$, $[g,h]$, $[i,j]$ and $[k,l]$ such that
$$
[g,h] \subseteq [a,b],\quad [i,j] \subseteq [c,d]\quad\text{and}\quad [k,l] \subseteq [e,f],
$$
and $\mathfrak{h}[a,b]+\mathfrak{h}[c,d]+\mathfrak{h}[e,f]=1$, if and only if the following facts hold:

\begin{enumerate}[itemsep=1.0ex, topsep=1.0ex, label=\emph{\arabic*)}]
\item For all $a\prec b$ in $X$ there exists a constant $C(a,b)\in K^{\times}$, such that
\begin{equation*}
\frac{\alpha_l(s)}{\alpha_l^{(1)}(s)}=\frac{\alpha_r(s)}{\alpha_r^{(1)}(s)}=C(a,b) ,\quad\text{for all $s\in X$.}
\end{equation*}

\item For all $a\prec b$ and $s$ in $X$, it is true that
$$
\qquad\qquad\! \alpha_r(s)=\alpha_r(s^{(2)}),\!\!\quad\alpha_r(s)=\alpha_r(s^{(2)}),\!\!\quad\beta_r(s)=\beta_r(s^{(2)}),\!\!\quad \beta_l(s)=\beta_l(s^{(2)}).
$$

\item For all $a\prec b$ in $X$ and each $0\le i< n$, there exists a one dimensional vector subspace of $K\times K$, which contains all the vectors
$$
\qquad\quad\biggl(\gamma\alpha_r(s)-w^i,\sum_{j=0}^{n-1}\wp_j w^{ij}\beta_r^{(j)}(s)\biggr)\quad\text{and}\quad
\biggl(\gamma\alpha_l(s)-w^i,\sum_{j=0}^{n-1}\wp_j w^{ij}\beta_l^{(j)}(s)\biggr),
$$
where $\gamma\coloneqq\gamma_r=\gamma_l$ and $\wp_j=\ell_j$ are as in Proposition~\ref{parte lineal}.
\end{enumerate}
\end{proposition}

\begin{proof}
We know that if $\mathfrak{h}[a,b]+\mathfrak{h}[c,d]+\mathfrak{h}[e,f]=1$, then equality~\eqref{eq braided} is equivalent to items~2), 3), 5), 6), 8) and~9) of Subsection~\ref{small intervals}. Moreover item~6) of Subsection~\ref{small intervals} is satisfied if and only if for all $a\prec b$ in $X$ there exists $C(a,b)\in K^{\times}$ such that
\begin{equation*}\label{cocientes alpha mas condicion nueve}
\frac{\alpha_l(s^{(1)})(a^{(1)},b^{(1)})}{\alpha_l(s)(a,b)}=\frac{\alpha_r({}^{(1)}s)({}^{(1)}\!a,{}^{(1)}\!b)}{\alpha_r(s)(a,b)} = C(a,b) \quad \text{for all $s\in X$.}
\end{equation*}
On the other hand Proposition~\ref{parte lineal} gives necessary and sufficient conditions in order that items~2), 3), 8) and~9) of Subsection~\ref{small intervals} are satisfied. Since $\phi_r = \phi_l$, we have $s^{(i)} = {}^{(i)} s$ for all $s\in X$ and all $i\in \mathds{Z}$, and $C_r(a,b) = C(a,b) = C_l(a,b)$ for all $a\prec b$ in~$X$. Consequently $\gamma_r$ and $\gamma_l$ are $n$th roots of the same element, and so we can choose $\gamma_r = \gamma_l$. It follows that $\wp_j=\ell_j$ for $0\le j<n$ and that the conditions in Proposition~\ref{parte lineal} are equivalent to items~1) and~2) together with the fact that there exist two one dimensional vector subspaces of $K\times K$ that contain all the vectors
$$
\biggl(\gamma\alpha_r(s)-w^i,\sum_{j=0}^{n-1}\wp_j w^{ij}\beta_r^{(j)}(s)\biggr)\quad\text{and}\quad \biggl(\gamma\alpha_l(s)-w^i,\sum_{j=0}^{n-1}\wp_j w^{ij}\beta_l^{(j)}(s)\biggr),
$$
respectively. Assume now that item~5) of Subsection~\ref{small intervals} is satisfied. Since $\phi_r = \phi_l$, using the equality in that item with $c\coloneqq s^{(i)}$, $(a,b)\coloneqq(a\hs^{(i)},b\hs^{(i)})$ and $e\coloneqq t^{(i)}$, where $i$ runs on $\{0,\dots,n-1\}$, we obtain
\begin{equation}\label{igualdades mixtas}
\begin{aligned}
\beta_l^{(0)}(s)+\alpha_l^{(0)}(s)\beta_r^{(1)}(t) & =\beta_r^{(0)}(t)+\alpha_r^{(0)}(t)\beta_l^{(1)}(s)  \\
\beta_l^{(1)}(s)+\alpha_l^{(1)}(s)\beta_r^{(2)}(t) &=\beta_r^{(1)}(t)+\alpha_r^{(1)}(t)\beta_l^{(2)}(s)  \\
&\vdotswithin{=}\\
\beta_l^{(n-1)}(s)+\alpha_l^{(n-1)}(s)\beta_r^{(0)}(t) &=\beta_r^{(n-1)}(t)+\alpha_r^{(n-1)}(t)\beta_l^{(0)}(s),
\end{aligned}
\end{equation}
where in the last equation we have used that $\beta_r^{(n)}=\beta_r^{(0)}$ and $\beta_l^{(n)}=\beta_l^{(0)}$. Mimicking the proof of item~3) of Proposition~\ref{parte lineal} we obtain that the equalities in~\eqref{igualdades mixtas} hold if and only if for $0\le i<n$
\begin{equation}\label{equivalencia mixta}
(\gamma \alpha_l(s)-w^i,S_l^{(i)}(s))\sim (\gamma \alpha_r(t)-w^i,S_r^{(i)}(t))\quad\text{for all $s,t\in X$},
\end{equation}
where $S_l^{(i)}(x)=\sum_{j=0}^{n-1}w^{ij} \wp_j\beta_l^{(j)}(x)$ and $S_r^{(i)}(x)$ is as in the proof of Proposition~\ref{parte lineal}. So item~3) is true. We leave the proof of the converse to the reader.
\end{proof}

\section[The configuration $x \prec y$ when $r_{|}$ is the flip]{The configuration $\bm{x \prec y}$ when $\bm{r_{|}}$ is the flip}\label{section la configuracion x menor que y}
Let $(X,\le)$ and $D$ be as in Section~2, let $(D,r)$ be a non-degenerate braided set and let $x, y\in X$ such that $x \prec y$. In this section we determine all the possibilities for the coefficients $\lambda_{a_1|b_1|a_2|b_2}^{a_3|b_3|a_4|b_4}$ with $a_i,b_i\in\{x,y\}$ and $a_i\le b_i$, under the assumption that
$$
r_{|}(x,x)=(x,x),\quad r_{|}(x,y)=(y,x),\quad r_{|}(y,x)=(x,y)\quad\text{and}\quad r_{|}(y,y)=(y,y).
$$

Let $f(x,x)\coloneqq 0$, $f(x,y)\coloneqq 1$ and $f(y,y)\coloneqq 2$. We can codify the $81$ coefficients $\lambda_{a_1|b_1|a_2|b_2}^{a_3|b_3|a_4|b_4}$ in a $9\times9$ matrix $M$, setting
\begin{equation}\label{matriz M}
M_{i,j}=\lambda_{a_1|b_1|a_2|b_2}^{a_3|b_3|a_4|b_4}\qquad \parbox[c]{5cm}{if $i=3f(a_1,b_1)+f(a_2,b_2)+1$ and $j=3f(a_3,b_3)+f(a_4,b_4)+1$.}
\end{equation}

\begin{remark}\label{ceros en M} By Proposition~\ref{donde puede no anularse},
$$
M=
\begin{pmatrix*}[c]
 1 & \phantom{{}_1}\beta_1 & 0 & \phantom{{}_1}\beta_2 & \Gamma_1 & 0 & 0 & 0 & 0 \\
 0 & 0 & 0 & \phantom{{}_1}\alpha_2 & B_1 & 0 & 0 & 0 & 0 \\
 0 & 0 & 0 & \!-\beta_2 & \Gamma_2 & 0 & 1 &  \phantom{{}_1}\beta_4 & 0 \\
 0 & \phantom{{}_1}\alpha_1 & 0 & 0 & B_2 & 0 & 0 & 0 & 0 \\
 0 & 0 & 0 & 0 & \!\! A & 0 & 0 & 0 & 0 \\
 0 & 0 & 0 & 0 & B_3 & 0 & 0 &  \phantom{{}_1}\alpha_4 & 0 \\
 0 & \!-\beta_1 & 1 & 0 & \Gamma_3 &  \phantom{{}_1}\beta_3 & 0 & 0 & 0 \\
 0 & 0 & 0 & 0 & B_4 &  \phantom{{}_1}\alpha_3 & 0 & 0 & 0 \\
 0 & 0 & 0 & 0 & \Gamma_4 & \!-\beta_3 & 0 & \!-\beta_4 & 1
\end{pmatrix*},
$$
where
\begin{alignat*}{4}
& \alpha_1:=\lambda_{x|x|x|y}^{x|y|x|x}, &&\quad \alpha_2:=\lambda_{x|y|x|x}^{x|x|x|y}, &&\quad \alpha_3:=\lambda_{x|y|y|y}^{y|y|x|y},&&\quad  \alpha_4:=\lambda_{y|y|x|y}^{x|y|y|y},\\[1.5pt]
&\beta_1:=\lambda_{x|x|x|y}^{x|x|x|x},&&\quad  \beta_2:=\lambda_{x|y|x|x}^{x|x|x|x},&&\quad \beta_3:=\lambda_{y|y|x|y}^{x|x|y|y},&&\quad \beta_4:=\lambda_{y|y|x|y}^{x|x|y|y},\\[1.5pt]
& A\coloneqq  \lambda_{x|y|x|y}^{x|y|x|y}, &&&&&&\\[1.5pt]
&B_1 \coloneqq \lambda_{x|y|x|y}^{x|x|x|y},&&\quad B_2 \coloneqq \lambda_{x|y|x|y}^{x|y|x|x},&&\quad B_3 \coloneqq \lambda_{x|y|x|y}^{x|y|y|y},&&\quad B_4 \coloneqq  \lambda_{x|y|x|y}^{y|y|x|y},\\[1.5pt]
&\Gamma_1:=\lambda_{x|y|x|y}^{x|x|x|x}, &&\quad \Gamma_2 \coloneqq \lambda_{x|y|x|y}^{x|x|y|y}, &&\quad \Gamma_3 \coloneqq  \lambda_{x|y|x|y}^{y|y|x|x}, &&\quad \Gamma_4 \coloneqq \lambda_{x|y|x|y}^{y|y|y|y}.
\end{alignat*}
\end{remark}

\begin{remark}
By Proposition~\ref{lambdas como productos} we know that
\begin{align}
& A=\alpha_1 \alpha_3=\alpha_2\alpha_4,\label{relaciones para A}\\[1.5pt]
& B_1=\alpha_2 \beta_4, \quad B_2=\alpha_1 \beta_3,\quad B_3=-\alpha_4\beta_2,\quad B_4 =  -\alpha_3\beta_1,\label{relaciones para Bs}\\[1.5pt]
&\Gamma_2 = -\beta_2\beta_4\quad\text{and}\quad \Gamma_3= -\beta_1\beta_3,\label{relaciones para Gamas}
\end{align}
and by Remark~\ref{compatibilidad con epsilon} and Proposition~\ref{donde puede no anularse}, we know that
\begin{equation}\label{relacion entre los gammas}
\Gamma_4 = -(\Gamma_1 + \Gamma_2 + \Gamma_3).
\end{equation}
We will use these equalities (which in particular show that $\Gamma_1$, $\alpha_1$, $\alpha_2$, $\alpha_3$, $\alpha_4$, $\beta_1$, $\beta_2$, $\beta_3$ and $\beta_4$ determine completely $M$) without explicit mention, in order to determine the fifth column of $M$ in all the cases.
Moreover, by Remark~\ref{condicion para iso}, if $r$ is an isomorphism, then $\alpha_i\ne 0$ for all~$i$.
\end{remark}

\begin{remark} If some $\beta_i\ne0$, then by equalities~\eqref{alineados 1}, \eqref{alineados 2} and~\eqref{alineados 3} there exists an element $C\in K$ such that
\begin{equation} \label{alpha de beta}
\alpha_i-1=C \beta_i,\quad\text{for $i=1,2,3,4$.}
\end{equation}
\end{remark}

\begin{theorem} \label{Lambdas caso flip x prec y}
The following facts hold:

\begin{enumerate}[itemsep=1.1ex, topsep=1.1ex, label=\emph{\arabic*)}]

\item If $\beta_i=0$ for $i=1,2,3,4$, and $\Gamma_1=0$, then $M$ belongs to the family
$$
\qquad\qquad \begin{pmatrix}[1]
 1 & 0 & 0 & 0 & 0 & 0 & 0 & 0 & 0 \\
 0 & 0 & 0 & \phantom{{}_1}\alpha_2 & 0 & 0 & 0 & 0 & 0 \\
 0 & 0 & 0 & 0 & 0 & 0 & 1 &  0 & 0 \\
 0 & \phantom{{}_1}\alpha_1 & 0 & 0 & 0 & 0 & 0 & 0 & 0 \\
 0 & 0 & 0 & 0 & \alpha_1\alpha_3 & 0 & 0 & 0 & 0 \\
 0 & 0 & 0 & 0 & 0 & 0 & 0 &  \phantom{{}_1}\alpha_4 & 0 \\
 0 & 0 & 1 & 0 & 0 &  0 & 0 & 0 & 0 \\
 0 & 0 & 0 & 0 & 0 &  \phantom{{}_1}\alpha_3 & 0 & 0 & 0 \\
 0 & 0 & 0 & 0 & 0& 0 & 0 & 0 & 1
\end{pmatrix},
$$
where $\alpha_4 = \frac{\alpha_1\alpha_3}{\alpha_2}$, parameterized by $\alpha_1,\alpha_2,\alpha_3\in K^{\times}$.

\item If $\beta_i=0$ for $i=1,2,3,4$ and $\Gamma_1\ne 0$, then $M$ belongs to the family
$$
\qquad\qquad
\begin{pmatrix}[1]
 1 & 0 & 0 & 0 & \phantom{{}_1}\Gamma_1 & 0 & 0 & 0 & 0 \\
 0 & 0 & 0 & \phantom{{}_1}\alpha_1 & 0 & 0 & 0 & 0 & 0 \\
 0 & 0 & 0 & 0 & 0 & 0 & 1 & 0 & 0 \\
 0 & \phantom{{}_1}\alpha_1 & 0 & 0 & 0 & 0 & 0 & 0 & 0 \\
 0 & 0 & 0 & 0 & \alpha_1\alpha_3 & 0 & 0 & 0 & 0 \\
 0 & 0 & 0 & 0 & 0 & 0 & 0 & \phantom{{}_1}\alpha_3 & 0 \\
 0 & 0 & 1 & 0 & 0 & 0 & 0 & 0 & 0 \\
 0 & 0 & 0 & 0 & 0 & \phantom{{}_1}\alpha_3 & 0 & 0 & 0 \\
 0 & 0 & 0 & 0 & -\Gamma_1 & 0 & 0 & 0 & 1
\end{pmatrix},
$$
parameterized by $\alpha_1=\pm 1$, $\alpha_2=\pm 1$ and $\Gamma_1\in K^{\times}$.

\item If there exists $i$ such that $\beta_i\ne 0$ and $C = 0$ (see equality~\eqref{alpha de beta}), then either
$$
\qquad\qquad \beta_3=\beta_2\!\!\quad\text{and}\!\!\quad \beta_4=\beta_1, \!\qquad\text{or}\qquad\! \beta_3+\beta_4=\beta_1+\beta_2=0\!\!\quad\text{and}\!\!\quad \beta_3\ne -\beta_1.
$$

\smallskip

\noindent In the first case $M$ belongs to the family
$$
\qquad\qquad\begin{pmatrix}[0.96]
 1 &\phantom{{}_1} \beta_1 & 0 & \phantom{{}_1}\beta_2 & \phantom{{}_1}\Gamma_1 & 0 & 0 & 0 & 0 \\
 0 & 0 & 0 & 1 & \phantom{{}_1}\beta_1 & 0 & 0 & 0 & 0 \\
 0 & 0 & 0 & \!-\beta_2 & \!-\beta_1 \beta_2 & 0 & 1 & \phantom{{}_1}\beta_1 & 0 \\
 0 & 1 & 0 & 0 & \phantom{{}_1}\beta_2 & 0 & 0 & 0 & 0 \\
 0 & 0 & 0 & 0 & 1 & 0 & 0 & 0 & 0 \\
 0 & 0 & 0 & 0 & \!-\beta_2 & 0 & 0 & 1 & 0 \\
 0 & \!-\beta_1 & 1 & 0 & \!-\beta_1 \beta_2 & \phantom{{}_1}\beta_2 & 0 & 0 & 0 \\
 0 & 0 & 0 & 0 & \!-\beta_1 & 1 & 0 & 0 & 0 \\
 0 & 0 & 0 & 0 & 2 \beta_1 \beta_2-\Gamma_1 & \!-\beta_2 & 0 & \!-\beta_1 & 1 \\
\end{pmatrix},
$$
parameterized by $\Gamma_1,\beta_1,\beta_2\in K$ with $(\beta_1,\beta_2)\ne (0,0)$.

\smallskip

\noindent In the second case $M$ belongs to the family
$$
\qquad \qquad\begin{pmatrix}[1]
 1 & \phantom{{}_1}\beta_1 & 0 & -\beta_1 & \phantom{{}_1}\beta_1 \beta_3 & 0 & 0 & 0 & 0 \\
 0 & 0 & 0 & 1 & -\beta_3 & 0 & 0 & 0 & 0 \\
 0 & 0 & 0 & \phantom{{}_1}\beta_1 & -\beta_1 \beta_3 & 0 & 1 & \! -\beta_3 & 0 \\
 0 & 1 & 0 & 0 & \phantom{{}_1}\beta_3 & 0 & 0 & 0 & 0 \\
 0 & 0 & 0 & 0 & 1 & 0 & 0 & 0 & 0 \\
 0 & 0 & 0 & 0 & \phantom{{}_1}\beta_1 & 0 & 0 & 1 & 0 \\
 0 & \!-\beta_1 & 1 & 0 & -\beta_1 \beta_3 & \phantom{{}_1}\beta_3 & 0 & 0 & 0 \\
 0 & 0 & 0 & 0 &\! -\beta_1 & 1 & 0 & 0 & 0 \\
 0 & 0 & 0 & 0 & \phantom{{}_1}\beta_1 \beta_3 &\! -\beta_3 & 0 & \phantom{{}_1} \beta_3 & 1 \\
\end{pmatrix},
$$
parameterized by $\beta_1,\beta_3\in K$ with $\beta_1+\beta_3\ne 0$.

\item If $C\ne 0$ and some $\beta_i\ne 0$, then
\begin{equation}\label{ultimos casos}
\qquad\quad\beta_1(\alpha_1+1)(\beta_3\alpha_1-\beta_2-\Gamma_1C)=0.
\end{equation}

\begin{enumerate}[itemsep=1.4ex, topsep=1.4ex, label=\emph{\alph*)}]

\item[\emph{4a)}] If $C\ne 0$, some $\beta_i\ne 0$ and $\beta_1= 0$, then either
$$
\qquad\qquad \Gamma_1= \alpha_2\beta_4/C, \!\qquad\text{or}\qquad\! \beta_2=0\!\!\quad\text{and}\!\!\quad \Gamma_1\ne \beta_4/C.
$$

\smallskip

\noindent In the first case $M$ belongs to the family
$$
\qquad\qquad\begin{pmatrix}[1]
 1 & 0 & 0 & \phantom{{}_1}\beta_2 & \phantom{{}_1}\frac{\alpha_2 \beta_4}{C} & 0 & 0 & 0 & 0 \\
 0 & 0 & 0 & \phantom{{}_1}\alpha_2 &\phantom{{}_1} \alpha_2 \beta_4 & 0 & 0 & 0 & 0 \\
 0 & 0 & 0 & \!-\beta_2 & \!-\beta_2 \beta_4 & 0 & 1 & \phantom{{}_1}\beta_4 & 0 \\
 0 & 1 & 0 & 0 & \phantom{{}_1}\beta_3 & 0 & 0 & 0 & 0 \\
 0 & 0 & 0 & 0 &\phantom{{}_1} \alpha_2 \alpha_4 & 0 & 0 & 0 & 0 \\
 0 & 0 & 0 & 0 &\! -\alpha_4 \beta_2 & 0 & 0 & \phantom{{}_1}\alpha_4 & 0 \\
 0 & 0 & 1 & 0 & 0 & \phantom{{}_1}\beta_3 & 0 & 0 & 0 \\
 0 & 0 & 0 & 0 & 0 & \phantom{{}_1}\alpha_3 & 0 & 0 & 0 \\
 0 & 0 & 0 & 0 & \beta_2 \beta_4-\frac{\alpha_2 \beta_4}{C} &\! -\beta_3 & 0 &\! -\beta_4 & 1 \\
\end{pmatrix},
$$
where $\alpha_i=1+C\beta_i$ and $\beta_3=\beta_2+\alpha_2\beta_4$, parameterized by $C\in K^{\times}$ and $\beta_2,\beta_4\in K$ with $(\beta_2,\beta_4)\ne (0,0)$ such that $C\beta_i\ne -1$ for all~$i$.

\smallskip

\noindent In the second case $M$ belongs to the family
$$
\qquad\qquad\begin{pmatrix}[1]
 1 & 0 & 0 & 0 & \phantom{{}_1}\Gamma_1 & 0 & 0 & 0 & 0 \\
 0 & 0 & 0 & 1 & \phantom{{}_1}\beta_4 & 0 & 0 & 0 & 0 \\
 0 & 0 & 0 & 0 & 0 & 0 & 1 & \phantom{{}_1}\beta_4 & 0 \\
 0 & 1 & 0 & 0 & \phantom{{}_1}\beta_4 & 0 & 0 & 0 & 0 \\
 0 & 0 & 0 & 0 &\!\!\!\! -1 & 0 & 0 & 0 & 0 \\
 0 & 0 & 0 & 0 & 0 & 0 & 0 & \!\!\!\! -1 & 0 \\
 0 & 0 & 1 & 0 & 0 & \phantom{{}_1}\beta_4 & 0 & 0 & 0 \\
 0 & 0 & 0 & 0 & 0 & \!\!\!\!-1 & 0 & 0 & 0 \\
 0 & 0 & 0 & 0 &\! -\Gamma_1 & \!-\beta_4 & 0 & \! -\beta_4 & 1 \\
\end{pmatrix},
$$
parameterized by $\Gamma_1\in K\setminus \{-\beta_4^2/2\}$ and $\beta_4\in K^{\times}$.

\item[\emph{4b)}] If $C\ne 0$, $\beta_1\ne 0$ and $\alpha_1+1=0$, then either
 $$
\qquad\qquad \alpha_2=-1, \!\qquad\text{or}\qquad\! \alpha_2\beta_4 C-C^2 \Gamma_1=-2\!\!\quad\text{and}\!\!\quad \alpha_2 \ne -1.
$$

\smallskip

\noindent In the first case $M$ belongs to the family
$$
\qquad\qquad\begin{pmatrix}[0.96]
 1 & \phantom{{}_1}\beta_1 & 0 & \phantom{{}_1} \beta_1 &\phantom{{}_1} \Gamma_1 & 0 & 0 & 0 & 0 \\
 0 & 0 & 0 &\!\!\!\! -1 & \!-\beta_4 & 0 & 0 & 0 & 0 \\
 0 & 0 & 0 &\! -\beta_1 &\! -\beta_1 \beta_4 & 0 & 1 & \phantom{{}_1}\beta_4 & 0 \\
 0 & \!\!\!\! -1 & 0 & 0 & \!-\beta_3 & 0 & 0 & 0 & 0 \\
 0 & 0 & 0 & 0 & \!-\alpha_4 & 0 & 0 & 0 & 0 \\
 0 & 0 & 0 & 0 & 2 \!-\alpha_4\beta_1 & 0 & 0 & \phantom{{}_1}\alpha_4 & 0 \\
 0 & \!-\beta_1 & 1 & 0 & \!-\beta_1 \beta_3 &\phantom{{}_1} \beta_3 & 0 & 0 & 0 \\
 0 & 0 & 0 & 0 & \!-\alpha_3\beta_1 & \phantom{{}_1}\alpha_3 & 0 & 0 & 0 \\
 0 & 0 & 0 & 0 & \beta_1 (\beta_3+\beta_4)-\Gamma_1 & \!-\beta_3 & 0 & \!-\beta_4 & 1 \\
\end{pmatrix},
$$
where $\alpha_i= 1+C\beta_i=1-\frac{2 \beta_i}{\beta_1}$, parameterized by $\beta_1\in K^{\times}$, $\Gamma_1\in K$ and $\beta_3,\beta_4\in K\setminus\{\beta_1/2\}$.

\smallskip

\noindent In the second case $M$ belongs to the family
$$
\qquad\qquad\begin{pmatrix}[0.96]
 1 & \phantom{{}_1}\beta_1 & 0 & \phantom{{}_1}\beta_2 & \phantom{{}_1}\Gamma_1 & 0 & 0 & 0 & 0 \\
 0 & 0 & 0 & \phantom{{}_1}\alpha_2 & \phantom{{}_1}\alpha_2\beta_4 & 0 & 0 & 0 & 0 \\
 0 & 0 & 0 & \!-\beta_2 & \!-\beta_2 \beta_4 & 0 & 1 & \phantom{{}_1}\beta_4 & 0 \\
 0 & \!\!\!\!-1 & 0 & 0 & \!-\beta_3 & 0 & 0 & 0 & 0 \\
 0 & 0 & 0 & 0 & \phantom{{}_1}\alpha_2\alpha_4 & 0 & 0 & 0 & 0 \\
 0 & 0 & 0 & 0 & \!-\alpha_4\beta_2 & 0 & 0 & \phantom{{}_1}\alpha_4 & 0 \\
 0 & \!-\beta_1 & 1 & 0 & \!-\beta_1 \beta_3 & \phantom{{}_1}\beta_3 & 0 & 0 & 0 \\
 0 & 0 & 0 & 0 & 2 \beta_3-\beta_1 & \phantom{{}_1}\alpha_3 & 0 & 0 & 0 \\
 0 & 0 & 0 & 0 & \frac{1}{2} \beta_1 (-\beta_1+2 \beta_3+\beta_4) & \!-\beta_3 & 0 & \!-\beta_4 & 1 \\
\end{pmatrix},
$$
where $\Gamma_1=\frac{1}{2}(\beta_1^2-\beta_4 \beta_1+2 \beta_2 \beta_4)$, $\beta_3=\frac{2 \beta_2 \beta_4}{\beta_1}+\beta_1-\beta_2-\beta_4$ and  $\alpha_i=1-\frac{2 \beta_i}{\beta_1}$, parameterized by $\beta_1\in K^{\times}$, $\beta_2\in K\setminus \{\beta_1\}$ and $\beta_4\in K$ such that $\beta_i\ne \beta_1/2$ for all~$i$.

\item[\emph{4c)}] If $C\ne 0$, $\beta_1\ne 0$ and $\alpha_1+1\ne 0$, then $M$ belongs to the family
$$
\qquad\qquad\begin{pmatrix}[0.96]
 1 & \phantom{{}_1}\beta_1 & 0 & \phantom{{}_1}\beta_2 &\phantom{{}_1}\Gamma_1 & 0 & 0 & 0 & 0 \\
 0 & 0 & 0 & \phantom{{}_1}\alpha_2 & \phantom{{}_1}\alpha_2\beta_4 & 0 & 0 & 0 & 0 \\
 0 & 0 & 0 & \!-\beta_2 & \!-\beta_2 \beta_4 & 0 & 1 & \phantom{{}_1}\beta_4 & 0 \\
 0 & \phantom{{}_1}\alpha_1 & 0 & 0 & \phantom{{}_1}\alpha_1\beta_3 & 0 & 0 & 0 & 0 \\
 0 & 0 & 0 & 0 & \phantom{{}_1}\alpha_2\alpha_4 & 0 & 0 & 0 & 0 \\
 0 & 0 & 0 & 0 & \!-\alpha_4\beta_2  & 0 & 0 & \phantom{{}_1}\alpha_4 & 0 \\
 0 & \!-\beta_1 & 1 & 0 & \!-\beta_1 \beta_3 & \phantom{{}_1}\beta_3 & 0 & 0 & 0 \\
 0 & 0 & 0 & 0 & \!-\alpha_3\beta_1 & \phantom{{}_1}\alpha_3 & 0 & 0 & 0 \\
 0 & 0 & 0 & 0 & \frac{\beta_4 C \beta_2+\beta_2-\beta_3}{C} & \!-\beta_3 & 0 &\! -\beta_4 & 1 \\
\end{pmatrix},
$$
where $\alpha_i=\beta_iC+1$ for all~$i$,  $\beta_4=(\beta_1 - \beta_2 + \beta_3 + \beta_1 \beta_3 C)/(1 + \beta_2 C)$ and $\Gamma_1= \frac{\beta_3\alpha_1-\beta_2}{C}$, parameterized by $C\in K^{\times}$, $\beta_1\in K^{\times}\setminus \{-1/C\}$ and $\beta_2,\beta_3\in K \setminus \{-1/C\}$.
\end{enumerate}
\end{enumerate}
\end{theorem}

\begin{proof}
1) \enspace Assume that $\beta_i=0$ for $i=1,2,3,4$, and that $\Gamma_1=0$. Then by the above discussion $\Gamma_2=\Gamma_3=\Gamma_4=0$. So, $M$ depends on $\alpha_i\in K^{\times}$, that satisfy the condition $\alpha_1\alpha_3=\alpha_2\alpha_4$. So, we obtain the family in item~1).

\medskip

Before considering the other cases we derive the two equalities~\eqref{identidad con gamma 1} and~\eqref{identidad con gamma 2}. For $a=b=c=e=x$, $g=h=i=j=k=l=x$ and $d=f=y$, equal\-i\-ty~\eqref{eq braided} yields
\begin{equation*}\label{gamma1}
\sum_{ z,v\in [x,y] } \lambda_{x|x|x|y}^{x|z|x|x} \lambda_{x|x|x|y}^{x|v|x|x} \lambda_{x|z|x|v}^{x|x|x|x}
=\sum_{ z,v\in [x,y] } \lambda_{x|y|x|y}^{x|v|x|z} \lambda_{x|x|x|v}^{x|x|x|x} \lambda_{x|x|x|z}^{x|x|x|x}.
\end{equation*}
Expanding this equality and using that $A=\alpha_1\alpha_3$, $B_1=\alpha_2\beta_4$ and $B_2=\alpha_1\beta_3$, we obtain
\begin{equation*}
\beta_1^2+\beta_1\alpha_1\beta_1+\alpha_1 \beta_1\beta_2+\alpha_1\alpha_1\Gamma_1=\Gamma_1+\alpha_1\beta_3\beta_1+\alpha_2\beta_4\beta_1+
\alpha_1\alpha_3\beta_1\beta_1,
\end{equation*}
which we write as
\begin{equation} \label{identidad con gamma 1}
\Gamma_1(\alpha_1^2-1)=\beta_1(\alpha_1\beta_3+\alpha_2\beta_4+\alpha_1\alpha_3\beta_1-\beta_1-\beta_1\alpha_1-\alpha_1\beta_2).
\end{equation}
For $a=c=d=e=x$, $g=h=i=j=k=l=x$ and $b=f=y$, equality~\eqref{eq braided} yields
\begin{equation*}
\sum_{ y,v\in [x,y] } \lambda_{x|y|x|x}^{x|x|x|y} \lambda_{x|y|x|y}^{x|v|x|x} \lambda_{x|x|x|v}^{x|x|x|x}
=\sum_{ y,v\in [x,y] } \lambda_{x|x|x|y}^{x|v|x|x} \lambda_{x|y|x|v}^{x|x|x|y} \lambda_{x|y|x|x}^{x|x|x|x}.
\end{equation*}
Expanding this equality and using that $B_1=\alpha_2\beta_4$ and $B_2=\alpha_1\beta_3$, we obtain
$$
\beta_2\beta_1+\beta_2\alpha_1\beta_1+\alpha_2\Gamma_1+\alpha_2(\alpha_1\beta_3)\beta_1=
\beta_1\beta_2+\alpha_1\Gamma_1+\beta_1\alpha_2\beta_2+\alpha_1(\alpha_2\beta_4)\beta_2,
$$
which simplifies to
\begin{equation}\label{identidad con gamma 2}
\alpha_1(\beta_1\beta_2-\Gamma_1)=\alpha_2(\beta_1\beta_2-\Gamma_1+\alpha_1(\beta_2\beta_4-\beta_1\beta_3)).
\end{equation}

\smallskip

\noindent 2)\enspace Assume that $\beta_i=0$ for $i=1,2,3,4$ and $\Gamma_1\ne 0$. Then equality~\eqref{identidad con gamma 1} implies that $\alpha_1^2=1$ and equality~\eqref{identidad con gamma 2} implies that $\alpha_1=\alpha_2$. Furthermore, similar calculations as above, using equality~\eqref{eq braided} with $b=c=d=f=y$, $g=h=i=j=k=l=y$ and $a=e=x$, give $\alpha_3=\alpha_4$, and using equality~\eqref{eq braided} with $a=b=d=f=y$, $c=e=x$ and $g=h=i=j=k=l=y$, give $\alpha_4^2=1$. Moreover, by~equalities~\eqref{relaciones para Bs} and~\eqref{relaciones para Gamas} we have
$$
B_1=B_2=B_3=B_4=\Gamma_2=\Gamma_3=0.
$$
So, in this case we obtain for $M$ the family in item 2).

\smallskip

\noindent 3)\enspace  Assume that some $\beta_i\ne0$ and $C=0$. Then $\alpha_j=1$ for all $j$ and equalities~\eqref{identidad con gamma 1} and~\eqref{identidad con gamma 2} yield
$$
\beta_1(\beta_3+\beta_4-\beta_1-\beta_2)=0\quad\text{and}\quad \beta_2\beta_4-\beta_1\beta_3=0.
$$
Moreover, a computation using equality~\eqref{eq braided} with $a=c=e=f=x$,  $b=d=y$ and $g=h=i=j=k=l=x$, shows that
\begin{equation*}\label{multiplicacion por beta 2}
\beta_2(\beta_1+\beta_2-\beta_3-\beta_4)=0;
\end{equation*}
a computation using equality~\eqref{eq braided} with $a=b=d=f=y$, $c=e=x$ and $g=h=i=j=k=l=y$, gives
\begin{equation*}\label{multiplicacion por beta 4}
\beta_4(\beta_1+\beta_2-\beta_3-\beta_4)=0;
\end{equation*}
and a computation using equality~\eqref{eq braided} with $b=d=e=f=y$, $a=c=x$ and $g=h=i=j=k=l=y$, gives
\begin{equation*}\label{multiplicacion por beta 3}
\beta_3(-\beta_1-\beta_2+\beta_3+\beta_4)=0.
\end{equation*}
Since at least one  $\beta_i$ is non zero, from these facts it follows that
\begin{equation}\label{caso tres}
\beta_3+\beta_4=\beta_1+\beta_2\quad\text{and}\quad \beta_2\beta_4-\beta_1\beta_3=0.
\end{equation}
A straightforward computation using \eqref{caso tres} shows that $(\beta_3-\beta_2)(\beta_1+\beta_2)=0$, and so either
$$
\beta_3=\beta_2\!\!\quad\text{and}\!\!\quad \beta_4=\beta_1, \!\qquad\text{or}\qquad\! \beta_3+\beta_4=\beta_1+\beta_2=0\!\!\quad\text{and}\!\!\quad \beta_3\ne -\beta_1.
$$
In the first case we obtain for M the first family in item~3). In the second case equality~\eqref{eq braided} with $a=c=e=x$, $b=d=f=y$ and $g=h=i=j=k=l=x$, gives
$$
2(\beta_1+\beta_3)(\beta_1\beta_3-\Gamma_1)=0.
$$
Since $\beta_3 \ne -\beta_1$, $\Gamma_1=\beta_1\beta_3$ and we obtain for $M$ the second family in item~3).

\smallskip

\noindent 4) \enspace Assume that $C\ne 0$ and some $\beta_i\ne0$. Hence, by equalities~\eqref{alpha de beta} and the fact that~$\alpha_1 \alpha_3=\alpha_2\alpha_4$, we have
$$
\alpha_2\beta_4 - \beta_1=\alpha_2\frac{\alpha_4-1}{C}-\beta_1=\frac{\alpha_1\alpha_3-\alpha_2}{C}-\beta_1=C\beta_1\beta_3 +\beta_3-\beta_2.
$$
Using this equalities, \eqref{alpha de beta} and~\eqref{identidad con gamma 1}, we obtain that
\begin{align*}
\Gamma_1 C\beta_1(\alpha_1+1) & = \beta_1(\alpha_1\beta_3+\alpha_2\beta_4+\alpha_1\alpha_3\beta_1-\beta_1-\beta_1\alpha_1-\alpha_1\beta_2)\\
&=\beta_1\bigl((1+C\beta_1)\beta_3+C\beta_1\beta_3 +\beta_3-\beta_2 +(1+C\beta_1)(1+C\beta_3)\beta_1\\
&-\beta_1(1+C\beta_1)-(1+C\beta_1)\beta_2\bigr)\\
&= \beta_1(C^2\beta_1^2\beta_3+3C\beta_1\beta_3-C\beta_1\beta_2+2\beta_3-2\beta_2)\\
&= \beta_1(C\beta_1+2)(\beta_3+C\beta_1\beta_3-\beta_2)\\
&=\beta_1(\alpha_1+1)(\beta_3\alpha_1-\beta_2),
\end{align*}
and so
\begin{equation*}
\beta_1(\alpha_1+1)(\beta_3\alpha_1-\beta_2-\Gamma_1C)=0.
\end{equation*}

\smallskip

4a) \enspace Assume that $C\ne 0$, some $\beta_i\ne 0$ and $\beta_1= 0$. Then  $\alpha_1=1$ and equality~\eqref{identidad con gamma 2} reduces to $-\Gamma_1=\alpha_2(-\Gamma_1+\beta_2 \beta_4)$. Using this and~\eqref{alpha de beta}, we obtain that
$$
\beta_2(C\Gamma_1-\alpha_2\beta_4)= \beta_2C\Gamma_1 + \Gamma_1 - \alpha_2\Gamma_1 =0.
$$
So $\Gamma_1= \alpha_2\beta_4/C$ or $\beta_2=0$. If $\Gamma_1= \alpha_2\beta_4/C$, then using that $\alpha_1=1$ and equal\-it\-ies~\eqref{relaciones para A} and~\eqref{alpha de beta}, we conclude that $\beta_3=\beta_2+\alpha_2\beta_4$. Hence, we obtain for~$M$ the first family in item~4a). Otherwise $\beta_2=0$ (which by~\eqref{relaciones para A} and \eqref{alpha de beta} implies that $\alpha_2=1$, $\alpha_3=\alpha_4$ and $\beta_3=\beta_4$) and $\Gamma_1\ne \beta_4/C$. Using now equality~\eqref{eq braided} with~ $c=e=x$, $a=b=d=f=y$ and $g=h=i=j=k=l=y$, we obtain that
$$
-\beta_4 (2 + \beta_4 C) (\beta_4 - C \Gamma_1)=0.
$$
But $\beta_4 - C \Gamma_1 \ne 0$, and  $\beta_4=0$ implies that $\beta_3=0$, which is impossible since at least one of the $\beta_i$'s is non zero. So we are left with $\beta_4\ne 0$ and $C=-2/\beta_4$, which yields for $M$ the second family in~4a).

\smallskip

4b) \enspace Assume that $C\ne 0$, $\beta_1\!\ne \!0$ and $\alpha_1+1\!=\!0$. Then $\beta_1\!=\!-2/C$ by equality~\eqref{alpha de beta}, and equal\-it\-y~\eqref{identidad con gamma 2} reads
$$
0=(\alpha_2+1)(\beta_1\beta_2-\Gamma_1)-\alpha_2(\beta_2\beta_4-\beta_1\beta_3).
$$
But since $\alpha_3=-\alpha_2\alpha_4$ by equality~\eqref{relaciones para A}, we have
$$
\beta_2\beta_4-\beta_1\beta_3=\frac{1}{C^2}((\alpha_2-1)(\alpha_4-1)+2(\alpha_3-1))=-\frac{1}{C^2}(\alpha_2+1)(\alpha_4+1),
$$
where the first equality holds by equality~\eqref{alpha de beta}, and therefore
\begin{equation}\label{eq: auxiliar}
0=(\alpha_2+1)\left(\beta_1\beta_2-\Gamma_1+\frac{1}{C^2}\alpha_2(\alpha_4+1)\right).
\end{equation}
Since
$$
C^2 \beta_1\beta_2+\alpha_2(\alpha_4+1)= -2\alpha_2 + 2 + \alpha_2(C \beta_4+2)= \alpha_2C \beta_4 + 2,
$$
because $C^2 \beta_1\beta_2=-2C\beta_2=-2\alpha_2+2$ and $\alpha_4+1=C \beta_4+2$, from equality~\eqref{eq: auxiliar} it follows that
$$
0=(\alpha_2+1)(2+\alpha_2\beta_4 C-C^2 \Gamma_1).
$$
So either $\alpha_2=-1$ or $\alpha_2\beta_4 C-C^2 \Gamma_1=-2$ and $\alpha_2\ne -1$. If~$\alpha_2=-1$, then we obtain for $M$ the first family in item~4b). Otherwise a direct computation using that $C=-2/\beta_1$ and equality~\eqref{alpha de beta}, shows that
$$
\Gamma_1=\frac{1}{2}(\beta_1^2-\beta_4 \beta_1+2 \beta_2 \beta_4),
$$
and a direct computation using that $C=-2/\beta_1$ and equalities~\eqref{identidad con gamma 1} and~\eqref{alpha de beta}, shows that
$$
\beta_3=\frac{2 \beta_2 \beta_4}{\beta_1}+\beta_1-\beta_2-\beta_4.
$$
So, we obtain for $M$ the second family in item~4b).

\smallskip

4c) \enspace Assume that $C\ne 0$, $\beta_1\ne 0$ and $\alpha_1+1\ne 0$. Then by equality~\eqref{ultimos casos}, we have~$\beta_3\alpha_1-\beta_2-\Gamma_1C=0$, and so $\Gamma_1= \frac{\beta_3\alpha_1-\beta_2}{C}$. Moreover, by equality~\eqref{alpha de beta} we know that $\alpha_i=\beta_iC+1$ for all~$i$, and using equalities~\eqref{relaciones para A} and \eqref{alpha de beta} it is easy to check that $\beta_4=(\beta_1 - \beta_2 + \beta_3 + \beta_1 \beta_3 C)/(1 + \beta_2 C)$. So, we obtain for $M$ the family in item 4c).
\end{proof}

\begin{corollary}\label{el orden x menor que y}
Let $X$ be the poset $(\{x,y\},\le)$, where $x<y$, let $D$ be the incidence coalgebra of $X$ and let $r\colon D \ot D \longrightarrow D\ot D$ be a map. If $(D,r)$ is a non-degenerate braided set, then $r_{|}$ is the flip and the matrix $M$ associated with $r$ via~\eqref{matriz M} belongs to one of the families in the previous theorem. On the other hand each member $M$ of the families yields a solution of the Yang-Baxter equation.
\end{corollary}

\begin{proof}
The first assertion follows immediately from Corollary~\ref{son iso de ordenes}, the second one is a corollary of Theorem~\ref{Lambdas caso flip x prec y}, and the third one follows by a direct computation, that can be done with the aid of a computer algebra system (set $M_1\coloneqq \id_D\ot M$ and $M_2\coloneqq M\ot \id_D$, and verify that $M_1M_2M_1-M_2M_1M_2=0$).
\end{proof}


\section[A case of the configuration $x \prec y\succ z$]{A case of the configuration $\bm{x \prec y\succ z}$}\label{seccion la configuracion x<y>z no flip}
Let $(X,\le)$ and $D$ be as in Section~2, let $(D,r)$ be a non-degenerate braided set and let $x, y,z\in X$ such that $x\prec y\succ z$.
Let $\phi$ be the permutation of $\{x,y,z\}$ that interchanges $x$ and $z$. In this section we determine all the possibilities for the coefficients
$\lambda_{a_1|b_1|a_2|b_2}^{a_3|b_3|a_4|b_4}$, with $a_i,b_i\in\{x,y,z\}$ and $a_i\le b_i$, under the assumption that~${}^a\hf(-)= (-)\hs^b=\phi$ for
all $a,b\in \{x,y,z\}$. Let $f(x,x):=0$, $f(x,y)\coloneqq 1$, $f(y,y)\coloneqq 2$, $f(z,y)\coloneqq 3$ and $f(z,z)\coloneqq 4$. We can codify the
$625$ coefficients $\lambda_{a_1|b_1|a_2|b_2}^{a_3|b_3|a_4|b_4}$ in a~$25\times25$ matrix $M$, setting
\begin{equation}\label{segunda Matriz M}
M_{i,j}=\lambda_{a_1|b_1|a_2|b_2}^{a_3|b_3|a_4|b_4}\qquad \parbox[c]{5cm}{if $i=5f(a_1,b_1)+f(a_2,b_2)+1$ and $j=5f(a_3,b_3)+f(a_4,b_4)+1$.}
\end{equation}
Let $\alpha_l$,  $\alpha_r$, $\beta_l$ and $\beta_r$ be the maps defined in~\eqref{definicion de a sub r} and~\eqref{definicion de a sub l}. We begin by showing that $M$ only depends on the entries
\begin{alignat*}{4}
&\Gamma_1 \coloneqq\lambda_{x|y|x|y}^{y|y|y|y},&&\quad  \Gamma_7\coloneqq\lambda_{x|y|z|y}^{y|y|y|y}, &&\quad \Gamma_{10}\coloneqq\lambda_{z|y|x|y}^{y|y|y|y},&&\quad \Gamma_{16}\coloneqq\lambda_{z|y|z|y}^{y|y|y|y},\\
& \alpha_1\coloneqq\alpha_l(x)(x,y),&&\quad \alpha_4\coloneqq\alpha_l(y)(x,y),&& \quad \alpha_6\coloneqq\alpha_l(z)(x,y),&& \\
&\beta_1\coloneqq\beta_l(x)(x,y),&& \quad\beta_2\coloneqq\beta_r(x)(x,y),&&\quad \beta_3\coloneqq\beta_r(y)(x,y),&&\quad
\beta_4\coloneqq\beta_l(y)(x,y),\\
&\beta_5\coloneqq\beta_r(z)(x,y),&& \quad\beta_6\coloneqq\beta_l(z)(x,y),&&\quad \beta_7\coloneqq\beta_l(x)(z,y),
&&\quad \beta_8\coloneqq\beta_r(x)(z,y),\\
&\beta_9\coloneqq\beta_l(y)(z,y),&&\quad \beta_{10}\coloneqq\beta_r(y)(z,y), &&\quad \beta_{11}\coloneqq\beta_r(z)(z,y),
&&\quad \beta_{12}\coloneqq\beta_l(z)(z,y)
\end{alignat*}
and the parameters
$$
C(x,y)\coloneqq \frac{\alpha_l(x)(x,y)}{\alpha_l^{(1)}(x)(x,y)}\quad\text{and}\quad C_2=\frac{\alpha_r(x)(x,y)}{\alpha_l(x)(x,y)},
$$
where $\alpha_l^{(1)}$ is as in~\eqref{a sub l a la i y a sub r a la i}. For this we first note that by Proposition~\ref{donde puede no anularse}, the matrix $M$ has the shape showed in Figure~\ref{matrix M en figura},
where $\Gamma_1$, $\Gamma_7$, $\Gamma_{10}$, $\Gamma_{16}$, $\alpha_1$, $\alpha_4$, $\alpha_6$, $\beta_1,\dots ,\beta_{12}$ are as above, and
\begin{alignat*}{4}
& A_1\coloneqq  \lambda_{x|y|x|y}^{z|y|z|y}, &&\quad A_2\coloneqq  \lambda_{x|y|z|y}^{x|y|z|y}, && \quad A_3\coloneqq  \lambda_{z|y|x|y}^{z|y|x|y}, &&\quad A_4\coloneqq  \lambda_{z|y|z|y}^{x|y|x|y},\\
&B_1 \coloneqq \lambda_{x|y|x|y}^{y|y|z|y}, &&\quad B_2 \coloneqq \lambda_{x|y|x|y}^{z|y|y|y},&&\quad B_3 \coloneqq \lambda_{x|y|x|y}^{z|y|z|z},&&\quad B_4 \coloneqq  \lambda_{x|y|x|y}^{z|z|z|y},\\
&B_5 \coloneqq \lambda_{x|y|z|y}^{x|x|z|y},&&\quad B_6 \coloneqq \lambda_{x|y|z|y}^{x|y|y|y},&&\quad B_7 \coloneqq \lambda_{x|y|z|y}^{x|y|z|z},&&\quad B_8 \coloneqq  \lambda_{x|y|z|y}^{y|y|z|y},\\
&B_9 \coloneqq \lambda_{z|y|x|y}^{y|y|x|y},&&\quad B_{10} \coloneqq \lambda_{z|y|x|y}^{z|y|x|x},&&\quad B_{11} \coloneqq \lambda_{z|y|x|y}^{z|y|y|y},&&\quad B_{12} \coloneqq  \lambda_{z|y|x|y}^{z|z|x|y},\\
&B_{13} \coloneqq \lambda_{z|y|z|y}^{x|x|x|y},&&\quad B_{14} \coloneqq \lambda_{z|y|z|y}^{x|y|x|x},&&\quad B_{15} \coloneqq \lambda_{z|y|z|y}^{x|y|y|y},&&\quad B_{16} \coloneqq  \lambda_{z|y|z|y}^{y|y|x|y},\\
& \Gamma_2 \coloneqq \lambda_{x|y|x|y}^{x|x|y|y},&&\quad \Gamma_3 \coloneqq  \lambda_{x|y|x|y}^{y|y|x|x},&&\quad \Gamma_4 \coloneqq \lambda_{x|y|x|y}^{y|y|y|y}, &&\quad \Gamma_5 \coloneqq \lambda_{x|y|z|y}^{x|x|y|y},\\
&\Gamma_6 \coloneqq  \lambda_{x|y|z|y}^{x|x|z|z},&&\quad \Gamma_8 \coloneqq \lambda_{x|y|z|y}^{y|y|z|z}, &&\quad \Gamma_9 \coloneqq \lambda_{z|y|x|y}^{y|y|x|x},&&\quad \Gamma_{11} \coloneqq  \lambda_{z|y|x|y}^{z|z|x|x},\\
&\Gamma_{12} \coloneqq \lambda_{z|y|x|y}^{z|z|y|y},&&\quad \Gamma_{13} \coloneqq \lambda_{z|y|z|y}^{x|x|x|x},&&\quad \Gamma_{14} \coloneqq  \lambda_{z|y|z|y}^{x|x|y|y},&&\quad \Gamma_{15} \coloneqq \lambda_{z|y|z|y}^{y|y|x|x},\\
&\alpha_2\coloneqq\alpha_r(x)(x,y),&&\quad \alpha_3\coloneqq\alpha_r(y)(x,y),&&\quad \alpha_5\coloneqq\alpha_r(z)(x,y),&&\\
&\alpha_7\coloneqq\alpha_l(x)(z,y),&&\quad \alpha_8\coloneqq\alpha_r(x)(z,y),&&\quad \alpha_9\coloneqq\alpha_l(y)(z,y),&&\\
&\alpha_{10}\coloneqq\alpha_r(y)(z,y),&&\quad \alpha_{11}\coloneqq\alpha_r(z)(z,y), &&\quad \alpha_{12}\coloneqq\alpha_l(z)(z,y).&&
\end{alignat*}

A direct computation using item~6) of Subsection~\ref{small intervals} proves that
\begin{equation}\label{alphas en terminos de alphas}
\alpha_k =\alpha_{13-k}/C(x,y) \quad\text{for $k=7,\dots,12$.}
\end{equation}
Moreover, by Proposition~\ref{lambdas como productos} we know that
\begin{align}
&\begin{alignedat}{2}\label{relaciones para As}
& A_1=\alpha_1 \alpha_3=\alpha_2\alpha_4,&&\qquad A_2=\alpha_3 \alpha_7=\alpha_5\alpha_9,\\
& A_3=\alpha_4 \alpha_8=\alpha_6\alpha_{10},&&\qquad A_4=\alpha_9 \alpha_{11}=\alpha_{10}\alpha_{12},
\end{alignedat}\\[3pt]
&\begin{alignedat}{4}\label{relaciones para los Bs}
& B_1=\alpha_2 \beta_4,  &&\qquad B_2=\alpha_1 \beta_3,&&\qquad B_3=-\alpha_4\beta_2,&&\qquad B_4 =  -\alpha_3\beta_1,\\
& B_5=\alpha_5 \beta_9, &&\qquad B_6=-\alpha_9 \beta_5,&&\qquad B_7=\alpha_7\beta_3,&&\qquad B_8 =  -\alpha_3\beta_7,\\
& B_9=-\alpha_{10} \beta_6,&& \qquad B_{10}=\alpha_6 \beta_{10},&&\qquad B_{11}=-\alpha_4\beta_8,&&\qquad B_{12} =  \alpha_8\beta_4,\\
& B_{13}=\alpha_{11} \beta_9,&& \qquad B_{14}=\alpha_{12} \beta_{10},&&\qquad B_{15}=-\alpha_9\beta_{11},&&\qquad B_{16} =  -\alpha_{10}\beta_{12}
\end{alignedat}
\shortintertext{and}
&\begin{alignedat}{4}\label{relaciones para Gamas caso 2}
&\Gamma_2 = -\beta_1\beta_3,&&\qquad \Gamma_3= -\beta_2\beta_4,&&\qquad\Gamma_5 = -\beta_5\beta_9,&&\qquad \Gamma_8= -\beta_7\beta_3,\\
&\Gamma_9 = -\beta_6\beta_{10},&&\qquad \Gamma_{12}= -\beta_8\beta_4,&&\qquad\Gamma_{14} = -\beta_{11}\beta_9,&&\qquad \Gamma_{15}= -\beta_{10}\beta_{12},
\end{alignedat}
\end{align}
and by Remark~\ref{compatibilidad con epsilon} and Proposition~\ref{donde puede no anularse}, we know that
\begin{equation}\label{relacion entre los gammas caso 2}
\begin{alignedat}{2}
&\Gamma_4 = -(\Gamma_1 + \Gamma_2 + \Gamma_3),&&\qquad \Gamma_6 = -(\Gamma_5 + \Gamma_7 + \Gamma_8),\\
&\Gamma_{11} = -(\Gamma_9 + \Gamma_{10} + \Gamma_{12}),&&\qquad \Gamma_{13} = -(\Gamma_{14} + \Gamma_{15} + \Gamma_{16}).
\end{alignedat}
\end{equation}
Equalities~\eqref{alphas en terminos de alphas}--\eqref{relacion entre los gammas caso 2} imply that
$\Gamma_1$, $\Gamma_7$, $\Gamma_{10}$, $\Gamma_{16}$, $\alpha_1,\dots ,\alpha_6$, $\beta_1,\dots ,\beta_{12}$ and $C(x,y)$
determine~$M$. A direct computation using equalities~\eqref{relaciones para As} and~\eqref{relaciones para los Bs} proves that
\begin{equation}\label{relaciones de Alphas}
\alpha_3 = C_2 \alpha_4\quad\text{and}\quad \alpha_5 = C_2 \alpha_6.
\end{equation}
So, $M$ only depends of $\Gamma_1$, $\Gamma_7$, $\Gamma_{10}$, $\Gamma_{16}$, $\alpha_1,\alpha_4 ,\alpha_6$, $\beta_1,\dots ,\beta_{12}$, $C(x,y)$ and~$C_2$,  as desired.

In the sequel we will provide without proofs analogous results to Theorem~\ref{Lambdas caso flip x prec y} and Corollary~\ref{el orden x menor que y},
for the configuration that we are considering.
Similar arguments as in the proof of Theorem~\ref{Lambdas caso flip x prec y} show that $M$ necessarily belongs to one of the families listed in Table~\ref{tabla}, where $C_1\in K$ is a fixed elements such that $C_1^2=1/C(x,y)$
and the elements $G_1,\dots, G_9$ and $F_j$, $j=1,\dots, 12$ are given by
\begin{align*}
&G_1 \coloneqq -\beta_1 \beta_3 C_1 + \beta_2 \beta_4 C_1 + \beta_1 \beta_5 C_1 - \beta_2 \beta_6 C_1 + \beta_5 C_4 - \beta_6 C_4 + \Gamma_{10}, \\[4pt]
&G_2 \coloneqq  -\beta_1^2 C_1^2 + \beta_1 \beta_4 C_1^2 - \beta_2 \beta_6 C_1^2 + \beta_3 \beta_6 C_1^2 + \beta_3 C_1 C_4 + \beta_4 C_1 C_4 + C_4^2 + C_1^2 \Gamma_1, \\[4pt]
&G_3 \coloneqq  -C_1(\beta_1+\beta_2), \\[4pt]
&G_4 \coloneqq  \frac{1}{4 C_1^2}\Bigl(-C_3^2(\alpha_1 \alpha_6 C_1^2-1) (\alpha_4 C_1 (\alpha_1 C_1 C_2+C_2+1)+1)\\
&\phantom{G_6} -2 C_3 C_4 (\alpha_1 \alpha_6 C_1^2+1) (\alpha_1 \alpha_4 C_1^2 C_2-1)\\
&\phantom{G_6} -C_4^2 (\alpha_1 \alpha_6 C_1^2-1) (\alpha_4 C_1 (C_2 (\alpha_1 C_1-1)-1)+1)+4 \alpha_1 \alpha_6 C_1^2 \Gamma_{16}\Bigr),\\[4pt]
&G_5 \coloneqq  -\frac{1}{4 C_1}\Bigl(-( C_2-1 ) (1 + \alpha_4 C_1 (1 + C_2 + \alpha_6 C_1 C_2)) C_3^2 - 2 \alpha_4 C_1 ( C_2^2-1 ) C_3 C_4\\
&\phantom{G7}  + ( C_2-1 ) (1 +  \alpha_4 C_1 ( \alpha_6 C_1 C_2 - C_2 -1)) C_4^2 - 4 C_1 C_2 \Gamma_{10}\Bigr),\\[4pt]
&G_6 \coloneqq  \frac 14 (C_3^2 + \alpha_4 C_1 C_3^2 + \alpha_4 C_1 C_2 C_3^2 + \alpha_1 \alpha_4 C_1^2 C_2 C_3^2 - 2 C_3 C_4 + 2 \alpha_1 \alpha_4 C_1^2 C_2 C_3 C_4 \\
&\phantom{G8} + C_4^2 - \alpha_4 C_1 C_4^2 - \alpha_4 C_1 C_2 C_4^2 + \alpha_1 \alpha_4 C_1^2 C_2 C_4^2),\\[4pt]
&G_7\coloneqq \frac 14 ((-C_3^2 - \alpha_4 C_1 C_3^2 - \alpha_4 C_1 C_2 C_3^2 - \alpha_4 \alpha_6 C_1^2 C_2 C_3^2 + 2 \alpha_4 C_1 C_3 C_4 \\
&\phantom{G9} -  2 \alpha_4 C_1 C_2 C_3 C_4 + C_4^2 - \alpha_4 C_1 C_4^2 - \alpha_4 C_1 C_2 C_4^2 + \alpha_4 \alpha_6 C_1^2 C_2 C_4^2))\\[4pt]
& F_j \coloneqq \alpha_j\frac{C_4-C_3}2-\frac{C_3+C_4}{2C_1}\quad\text{for $j=1,\dots,6$,}
\intertext{and}
& F_j \coloneqq  \alpha_{13-j} C_1\frac{C_3+C_4}2+\frac{C_3-C_4}2\quad\text{for $j=7,\dots,12$.}
\end{align*}

\begin{remark}
Let $X$ be the poset $(\{x,y,\},x\prec y \succ z)$, let $D$ be the incidence coalgebra of $X$ and let
$$
r\colon D \ot D \longrightarrow D\ot D
$$
be a map. By the same argument as in the proof of Corollary~\ref{el orden x menor que y}, if $(D,r)$ is a non-degenerate braided set such that $r_{|}$ is not the flip, then the matrix $M$ associated with $r$ via~\eqref{segunda Matriz M} belongs to one of the families in Table~\ref{tabla}. On the other hand each member $M$ of the families yields a solution of the Yang-Baxter equation. Note that these families are not disjoint.
\end{remark}

\setcounter{MaxMatrixCols}{25}
\begin{amssidewaysfigure}
\centering
$$
\begin{pmatrix*}[c]
 0 & 0 & 0 & 0 & 0 & 0 & 0 & 0 & 0 & 0 & 0 & 0 & 0 & 0 & 0 & 0 & 0 & 0 & \Gamma_{13} & \beta_{11} & 0 & 0 & 0 & \beta_{12} & 1 \\
 0 & 0 & 0 & 0 & 0 & 0 & 0 & 0 & 0 & 0 & 0 & 0 & 0 & 0 & 0 & 0 & 0 & 0 & B_{13} & \alpha_{11} & 0 & 0 & 0 & 0 & 0 \\
 0 & 0 & 0 & 0 & 0 & 0 & 0 & 0 & \Gamma_5 &\!\!\!\! -\beta_5 & 0 & 0 & 0 & \beta_9 & 1 & 0 & 0 & 0 & \Gamma_{14} & \!\!\!\!-\beta_{11} & 0 & 0 & 0 & 0 & 0\\
 0 & 0 & 0 & 0 & 0 & 0 & 0 & 0 & B_5 & \alpha_5 & 0 & 0 & 0 & 0 & 0 & 0 & 0 & 0 & 0 & 0 & 0 & 0 & 0 & 0 & 0 \\
 0 & 0 & 0 & \beta_7 & 1 & 0 & 0 & 0 & \Gamma_6 & \beta_5 & 0 & 0 & 0 & 0 & 0 & 0 & 0 & 0 & 0 & 0 & 0 & 0 & 0 & 0 & 0 \\
 0 & 0 & 0 & 0 & 0 & 0 & 0 & 0 & 0 & 0 & 0 & 0 & 0 & 0 & 0 & 0 & 0 & 0 & B_{14} & 0 & 0 & 0 & 0 & \alpha_{12} & 0 \\
 0 & 0 & 0 & 0 & 0 & 0 & 0 & 0 & 0 & 0 & 0 & 0 & 0 & 0 & 0 & 0 & 0 & 0 & A_4 & 0 & 0 & 0 & 0 & 0 & 0 \\
 0 & 0 & 0 & 0 & 0 & 0 & 0 & 0 & B_6 & 0 & 0 & 0 & 0 & \alpha_9 & 0 & 0 & 0 & 0 & B_{15} & 0 & 0 & 0 & 0 & 0 & 0 \\
 0 & 0 & 0 & 0 & 0 & 0 & 0 & 0 & A_2 & 0 & 0 & 0 & 0 & 0 & 0 & 0 & 0 & 0 & 0 & 0 & 0 & 0 & 0 & 0 & 0 \\
 0 & 0 & 0 & \alpha_7 & 0 & 0 & 0 & 0 & B_7 & 0 & 0 & 0 & 0 & 0 & 0 & 0 & 0 & 0 & 0 & 0 & 0 & 0 & 0 & 0 & 0 \\
 0 & 0 & 0 & 0 & 0 & 0 & 0 & 0 & 0 & 0 & 0 & 0 & 0 & 0 & 0 & 0 &\Gamma_9 &\beta_{10} & \Gamma_{15} & 0 & 0 &\!\!\!\! -\beta_6 & 1 &\!\!\!\!-\beta_{12} & 0\\
 0 & 0 & 0 & 0 & 0 & 0 & 0 & 0 & 0 & 0 & 0 & 0 & 0 & 0 & 0 & 0 & B_9 & \alpha_{10} & B_{16} & 0 & 0 & 0 & 0 & 0 & 0 \\
 0 & 0 & 0 & 0 & 0 & 0 & \Gamma_1 & \!\!\!\!-\beta_3 & \Gamma_7 & 0 & 0 & \!\!\!\!-\beta_4 & 1 & \!\!\!\!-\beta_9 & 0 & 0 & \Gamma_{10} & \!\!\!\!-\beta_{10} & \Gamma_{16} & 0 & 0 & 0 & 0 & 0 & 0 \\
 0 & 0 & 0 & 0 & 0 & 0 & B_1 & \alpha_3 & B_8 & 0 & 0 & 0 & 0 & 0 & 0 & 0 & 0 & 0 & 0 & 0 & 0 & 0 & 0 & 0 & 0 \\
 0 & \!\!\!\!-\beta_1 & 1 & \!\!\!\!-\beta_7 & 0 & 0 & \Gamma_2 & \beta_3 & \Gamma_8 & 0 & 0 & 0 & 0 & 0 & 0 & 0 & 0 & 0 & 0 & 0 & 0 & 0 & 0 & 0 & 0 \\
 0 & 0 & 0 & 0 & 0 & 0 & 0 & 0 & 0 & 0 & 0 & 0 & 0 & 0 & 0 & 0 & B_{10} & 0 & 0 & 0 & 0 & \alpha_6 & 0 & 0 & 0 \\
 0 & 0 & 0 & 0 & 0 & 0 & 0 & 0 & 0 & 0 & 0 & 0 & 0 & 0 & 0 & 0 & A_3 & 0 & 0 & 0 & 0 & 0 & 0 & 0 & 0 \\
 0 & 0 & 0 & 0 & 0 & 0 & B_2 & 0 & 0 & 0 & 0 & \alpha_4 & 0 & 0 & 0 & 0 & B_{11} & 0 & 0 & 0 & 0 & 0 & 0 & 0 & 0 \\
 0 & 0 & 0 & 0 & 0 & 0 & A_1 & 0 & 0 & 0 & 0 & 0 & 0 & 0 & 0 & 0 & 0 & 0 & 0 & 0 & 0 & 0 & 0 & 0 & 0 \\
 0 & \alpha_1 & 0 & 0 & 0 & 0 & B_3 & 0 & 0 & 0 & 0 & 0 & 0 & 0 & 0 & 0 & 0 & 0 & 0 & 0 & 0 & 0 & 0 & 0 & 0 \\
 0 & 0 & 0 & 0 & 0 & 0 & 0 & 0 & 0 & 0 & 0 & 0 & 0 & 0 & 0 & \beta_8 & \Gamma_{11} & 0 & 0 & 0 & 1 & \beta_6 & 0 & 0 & 0 \\
 0 & 0 & 0 & 0 & 0 & 0 & 0 & 0 & 0 & 0 & 0 & 0 & 0 & 0 & 0 & \alpha_8 & B_{12} & 0 & 0 & 0 & 0 & 0 & 0 & 0 & 0 \\
 0 & 0 & 0 & 0 & 0 & \!\!\!\!-\beta_2 & \Gamma_3 & 0 & 0 & 0 & 1 & \beta_4 & 0 & 0 & 0 & \!\!\!\!-\beta_8 & \Gamma_{12} & 0 & 0 & 0 & 0 & 0 & 0 & 0 & 0 \\
 0 & 0 & 0 & 0 & 0 & \alpha_2 & B_4 & 0 & 0 & 0 & 0 & 0 & 0 & 0 & 0 & 0 & 0 & 0 & 0 & 0 & 0 & 0 & 0 & 0 & 0 \\
 1 & \beta_1 & 0 & 0 & 0 & \beta_2 & \Gamma_4 & 0 & 0 & 0 & 0 & 0 & 0 & 0 & 0 & 0 & 0 & 0 & 0 & 0 & 0 & 0 & 0 & 0 & 0 \\
\end{pmatrix*}
$$
\caption{The matrix $M$}
\label{matrix M en figura}
\end{amssidewaysfigure}

\setlength{\tabcolsep}{3.9pt}
\ra{1.2}
\begin{table}
\begin{tabular}{lllll}
\toprule
$\#$ & \begin{minipage}[l]{2.3cm}\begin{center} Fixed values \\ in each family \end{center} \end{minipage} &
\!\!\!\begin{minipage}[l]{2.3cm}\begin{center} Fixed values \\ in subfamilies \end{center} \end{minipage}& \begin{minipage}[l]{2.3cm}\begin{center} Continuous \\ parameters \end{center} \end{minipage} & \begin{minipage}[l]{2.3cm}\begin{center} Discrete \\ parameters \end{center} \end{minipage}\\
\midrule
1. & For $j=1,\dots,12$ && $\alpha_1,\alpha_4,\alpha_6\in K^{\times}$ & \\
&$\beta_j=0$ &&$C_1, C_2\in K^{\times}$&\\[4pt]
& $\Gamma_{\!1}\!=\!\Gamma_{\!7}\!=\!\Gamma_{\!10}\!=\!\Gamma_{\!16}\!=\!0$ & & &\\
\midrule
2. & For $j=1,\dots,12$ & $\Gamma_{16}=0$ &$C_1,\Gamma_7\in K^{\times}$ & $C_2\in\{\pm 1\}$\\
   & $\beta_j=0$ & $\alpha_1=\varepsilon_1/C_1$ & & $\varepsilon^2=C_2$ \\
   &   & $\alpha_4=\varepsilon_4/C_1$ & & $\varepsilon_1,\varepsilon_4,\varepsilon_6\in\{\pm \varepsilon\}$ \\
   &   & $\alpha_6=\varepsilon_6/C_1$ & &\\
\cmidrule(r){3-5}
   &$\Gamma_1=\alpha_1\alpha_6\Gamma_{16}$ & $\alpha_6=\alpha_1$ &$C_1,\Gamma_{16}\in K^{\times}$ & $C_2\in\{\pm 1\}$\\
   & $\Gamma_{10}=C_2\Gamma_{7}$ & $\alpha_1=\varepsilon_1/C_1$ &  $\Gamma_7\in K$ & $\varepsilon^2=C_2$ \\
   &   & $\alpha_4=\varepsilon_4/C_1$ & & $\varepsilon_1,\varepsilon_4\in\{\pm \varepsilon\}$ \\
\midrule
3. & $\alpha_1=\alpha_4=\alpha_6=\frac{1}{C_1}$& For $j\!=\!2,\dots,6$ &$\Gamma_1,\Gamma_{10}\in K$&\\
   & $C_2=1$ & $\beta_j=\beta_1$ & $\beta_1,C_1\in K^{\times}$ & \\
   & & $C_4=G_3$ & &\\
\cmidrule(r){3-5}
   & For $j\!=\!1,\dots,6$ & $\beta_3=\beta_5=\beta_2$&$\Gamma_1,\Gamma_{10},C_4\in K$ &\\
   & $\beta_{13-j}=C_4+C_1\beta_j$ & $\beta_4=\beta_6=\beta_1$ &$C_1\in K^{\times}$&\\
   & & &$\beta_1,\beta_2\!\in \!K$, $\beta_1\!\ne \!\beta_2$ &\\
\cmidrule(r){3-5}
   & $\Gamma_7=G_1$ & $\Gamma_1=\beta_1 \beta_2$ &$C_1\in K^{\times}$&\\
   & $\Gamma_{16}=G_2$ & $\Gamma_{10}=-\beta_1\beta_6 C_1 $&$\beta_1,\beta_2,\beta_5\in K$,  $\beta_2\!\ne \!\beta_5$&\\
   & & $C_4=G_3$  & &\\
   & &$\beta_4=\beta_1$ & &\\
   & &$\beta_3=\beta_2$ & &\\
   & & $\beta_{\!6}=\beta_{\!1}\!+\!\beta_{\!2}\! - \!\beta_{\!5}$  & &\\
\midrule
4. & For $j=1,\dots,12$  & $\Gamma_{10}=G_7$ & $\alpha_1,\alpha_4,\alpha_6\in K^{\times}$&\\
   & $\beta_j=F_j$ & & $C_1, C_2\in K^{\times}$&\\
&  &&$C_3, C_4\in K$&\\
\cmidrule(r){3-5}
  & $\Gamma_1=G_4$ & $\alpha_1=\varepsilon_1/C_1$ &$C_1\in K^{\times}$ & $C_2\in\{\pm 1\}$\\
  & $\Gamma_7=G_5$ &$\alpha_4=\varepsilon_4/C_1$&$C_3, C_4,\Gamma_{10}\in K$& $\varepsilon^2=C_2$ \\
  & $\Gamma_{16}=G_6$ & $\alpha_6=\varepsilon_6/C_1$ &&$\varepsilon_1,\varepsilon_4,\varepsilon_6\in\{\pm \varepsilon\}$\\
\bottomrule
\end{tabular}
\vspace{5pt}
\caption{Families for $M$}
  \label{tabla}
\end{table}

\end{document}